\documentclass[12pt,reqno]{amsart}
\usepackage{amsmath}
\usepackage{amssymb}
\usepackage{amsfonts}
\usepackage{mathrsfs}
\usepackage{graphicx}
\usepackage{epsfig}
\usepackage{enumitem}
\usepackage[T1]{fontenc}
\numberwithin{equation}{section}       % Number formulas within sections
\usepackage{longtable}
\usepackage{threeparttable}

\usepackage{booktabs}
\usepackage[font=normalfont,tableposition=top,figureposition=bottom,skip=0pt]{caption}
\usepackage[utf8]{inputenc} % Required for including letters with accents
\usepackage[T1]{fontenc} % Use 8-bit encoding that has 256 glyphs
\usepackage{lmodern}
\usepackage{xcolor}

\theoremstyle{plain}
\newtheorem{theorem}{Theorem}[section]

\newtheorem{prop}{Proposition}[section]

\newtheorem{coro}[prop]{Corollary}
\newtheorem{lemma}[prop]{Lemma}

\usepackage{hyperref}

\theoremstyle{definition}
\newtheorem{definition}[prop]{Definition}

\theoremstyle{remark}
\newtheorem{remark}[prop]{Remark}

\newtheoremstyle{citing}% name
  {3pt}%      Space above, empty = `usual value'
  {3pt}%      Space below
  {\itshape}% Body font
  {}%         Indent amount (empty = no indent, \parindent = para indent)
  {\bfseries}% Thm head font
  {.}%        Punctuation after thm head
  {.5em}%     Space after thm head: " " = normal interword space;
  {\thmnote{#3}}% Thm head spec

\theoremstyle{citing}
\DeclareMathAlphabet{\mathpzc}{OT1}{pzc}{m}{it} % Zapf Chancery math alphabet

\newcommand{\hyp}{\textrm{hyp}}

\newcommand{\C}{\mathbb{C}}

\newcommand{\N}{\mathbb{N}}

\newcommand{\Q}{\mathbb{Q}}
\newcommand{\R}{\mathbb{R}}
\newcommand{\T}{\mathbb{T}}

\newcommand{\Z}{\mathbb{Z}}

\newcommand{\teta}{\widetilde{\teta}}

\newcommand{\eps}{\varepsilon}
\newcommand{\dist}{d}

\DeclareMathOperator{\diam}{diam}

\DeclareMathOperator{\supp}{supp} % Support of a measure
\def\marginparQ#1{\marginpar{}}
\begin{document}
\title[]{$L^\infty$-estimation of generalized Thue-Morse trigonometric polynomials and ergodic maximization}
%\author{}

\author{Aihua FAN}
\address{
	%School of Mathematics and Statistics, Central China Normal %University,430079, Wuhan, China \&
	LAMFA, UMR 7352 CNRS, University
	of Picardie, 33 rue Saint Leu, 80039 Amiens, France}
\email{ai-hua.fan@u-picardie.fr}

%	\urladdr{www.math.sc.edu/$\sim$howard}
\author{J\"{O}RG SCHMELING}
\address{Lund University\\
	Centre for Mathematical Sciences\\
	Box 118, 221 00 LUND,  Sweden}
\email{joerg@maths.lth.se}

\author{Weixiao SHEN}
\address{Shanghai Center for Mathematical Sciences\\ Fudan University\\ 220 Handan Road, Shanghai 200433, China}
\email{wxshen@fudan.edu.cn}

%\date{\today}

\maketitle
%\marginpar{rephrased}
\begin{abstract} 
%Generalized Thue-Morse sequences
	%$(t_n^{(q;c)})_{n\ge 0}$ ($q\ge 2$ and $c \in [0,1)$ being  parameters)
	%are defined by
	%$
	%t_n^{(q;c)} = e^{2\pi i c S_q(n)}
	%$, where $S_q(n)$ is the sum of  digits of the $q$-expansion of $n$, where $q\ge 2$ is an integer. 
Given an integer $q\ge 2$ and a real number $c\in [0,1)$, consider the generalized Thue-Morse sequence 
$(t_n^{(q;c)})_{n\ge 0}$ defined by $t_n^{(q;c)} = e^{2\pi i c S_q(n)}$, where $S_q(n)$ is the sum of  digits of the $q$-expansion of $n$.
We prove that %For 
the $L^\infty$-norm of the trigonometric polynomials
	$\sigma_{N}^{(q;c)} (x) := \sum_{n=0}^{N-1} t_n^{(q;c)} e^{2\pi i n x}$,
	%we prove that the uniform norm $\|\sigma_N^{(c)}\|_\infty$
	behaves like $N^{\gamma(q;c)}$, 
 %and the best exponent 
 where $\gamma(q;c)$ is equal to the dynamical
	maximal value of $\log_q \left|\frac{\sin q\pi (x+c)}{\sin \pi (x+c)}\right|$ relative to the
	dynamics $x \mapsto qx \mod 1$ and that the maximum value is attained
	by a $q$-Sturmian measure. Numerical values of $\gamma(q;c)$ can be computed.
\iffalse	
	Minimization is also discussed.
	%When $\xi: = 1/2 -c$
	%is a periodic point,
	It is also proved that $2^{-n}|\sigma_{2^n}^{(c)}(x)|$ behaves pointwise like $e^{n \alpha (x)}$
	with $\alpha(x) <0$ and that the function
	$\alpha(x)$ is multifractal.
\fi

%.\cite{bassily2009note},\cite{delange1972fonctions},\cite{grabner1993completely},\cite{liardet1987regularities}
\end{abstract}

\tableofcontents

\section {Introduction and main results}

Let $q\ge 2$ be a positive integer. For any integer $n\ge 0$, we denote by $S_q(n)$  the sum of  digits of expansion of $n$ in base $q$. Fix $c\in [0,1)$, we define the generalized Thue-Morse sequence $(t_{n}^{(q;c)})_{n\ge 0}$ by
$$
        t_{n}^{(q;c)} = e^{2\pi i c S_q(n)}.
$$
The case that $q=2$ and $c=1/2$ corresponds to the classical Thue-Morse sequence:
$$
 1, -1, -1, 1,-1, 1, 1, -1, -1, 1, 1, -1, 1, -1, -1, 1, \cdots.
$$
%\marginpar{deleted}%which will be
%simply written as $(t_{n})$.~\marginpar{delete?}
%The case $c=0$, which give the constant sequence $t_n^{(0)} =1$, will be %excluded so that we will assume that $0<c<1$.
By a {\em generalized Thue-Morse trigonometric series} we mean
$$
   \sum_{n=0}^{\infty} t_{n}^{(q;c)} e^{2\pi i n x},
$$
which defines a distribution on the circle $\mathbb{T}:=\mathbb{R}/\mathbb{Z}$. %In this paper,
We are interested in the asymptotic behaviors of its partial sums, called the {\em generalized Thue-Morse  trigonometric
polynomials}:
\begin{equation}\label{GTM}
     \sigma_{N}^{(q;c)} (x) := \sum_{n=0}^{N-1} t_{n}^{(q;c)} e^{2\pi i n x} \quad (N\ge 1).
\end{equation}
%We shall study the asymptotic behaviors of $\sigma_{N}^{(q;c)} (x)$
%by investigating the  growth of the uniform norm $\|\sigma_{N}^{(q;c)}\|_\infty$ and the multifractal analysis of  the limit $\lim_{n\to\infty} %n^{-1}\log |\sigma_{q^n}^{(q;c)} (x)|$.

The first problem is to find or to estimate the best constant $\gamma$ such that
\begin{equation}\label{eqn:BN}
\sup_{x\in\R}\left|\sum_{n=0}^{N-1} t_{n}^{(q;c)} e^{2\pi i n x}\right|=O(N^\gamma).
\end{equation}
Define $\gamma(q, c)$, sometimes denoted $\gamma(c)$, to be the infimum of all $\gamma$ for which (\ref{eqn:BN}) holds.
Following Fan \cite{F}, we call $\gamma(c)$ the {\em Gelfond exponent}
of the generalized Thue-Morse sequence $(t^{(q;c)}_{n})$.
%We call $\gamma(c)$ the Gelfond exponent of $(t_n^{(c)})_{n\ge 0}$.
%Clearly $\gamma(c)\ge 1/2$ since $  \|\sigma_N^{(c)}\|_2 =N$.
 The first result, due to Gelfond \cite{Gelfond1968}, is that
 $$
 \gamma(2; 1/2)= \frac{\log 3}{\log 4} =0.792481....
 $$
 Trivially $\gamma(q; 0)=1$. No other exact exponents $\gamma(q; c)$ are known.
 A basic fact, as %~\marginpar{changed} 
 a consequence of the so-called $q$-multiplicativity of
 $(t^{(q;c)}_n)$,  is the following expression
 \begin{equation}
     |\sigma_{q^n}^{(c)}(x)| = \prod_{k=0}^{n-1} \left|\frac{\sin \pi (q^k x +c)}{\sin \pi(x+c)}\right|.
 \end{equation}

% \marginpar{from here check later}
 Thus the dynamical system  $T=T_q: \mathbb{T}\to \mathbb{T}$ defined by $Tx = q x \mod 1$ is naturally involved. Let
 $$
    f_{q;c}(x):= f_{c}(x) :=\log \left|\frac{\sin q\pi (x+c)}{\sin \pi (x+c)}\right|.
 $$
 We will simply write $f_c$ if there is no confusion.
Let us point out that $f_c$ is a translation of $f_0$
and that $f_0(x)\le \log q$ for all $x$ and $f_0(0)=\log q$, and $f_0$ has $q-1$ singularities as a function on $\mathbb{T}$ in the sense $f_0(r/q)=-\infty$ for $1\le k\le q-1$ . Furthermore, $f_0$
 is concave between any two adjacent singularity points.
  Consequently $f_c$ attains its maximal value at $x=-c$ and its singularity points are $b_k:= -c +k/q$ ($1\le k \le q-1$).
  See Figure \ref{fig:f0} for its graph.

 \begin{figure}[htb]
 	\centering
 	\includegraphics[width=0.80\linewidth]{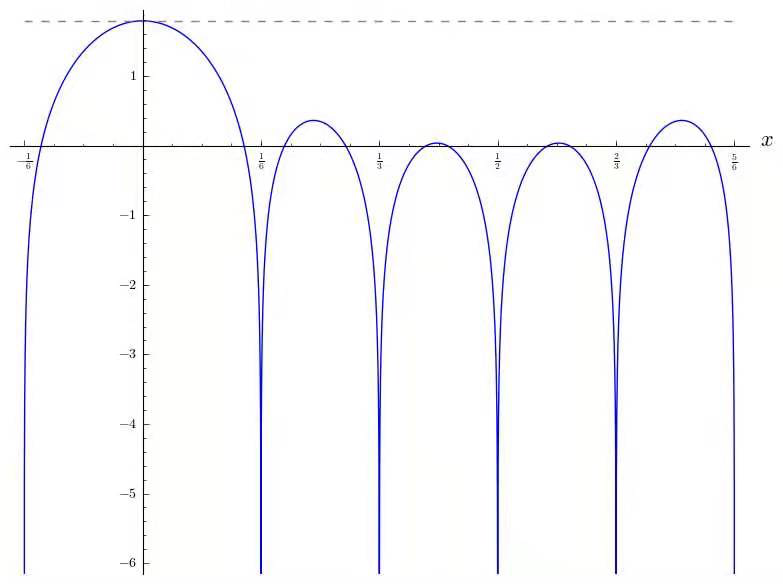}
 	\caption{The graphs of $f_0$ on the interval % $ [0,1] $ and
 		$ [-1/q,1-1/q] $, here $q=6$.}
 	\label{fig:f0}
 \end{figure}

 %where $b : = 1/2 - c \in (-1/2, 1/2]$ is the unique zero of the %function
 %$|\cos \pi (x+c)|$,
 %which will play a crucial role.

As we shall see in Proposition~\ref{prop:b-g}, finding the Gelfond exponent $\gamma(c)=\gamma(q;c)$ % ~\marginpar{added} 
is
equivalent to maximizing $f_c$. That is to say
\begin{equation}\label{Beta-Gamma}
\gamma (c) = \frac{\beta(c)}{\log q}
\end{equation}
with
\begin{equation}\label{Beta}
       \beta(c): = \sup_{\mu \in \mathcal{M}_T} \int_{\mathbb{T}}
         f_c(x) d\mu(x)
\end{equation}
where $\mathcal{M}_T$ is the set of $T$-invariant Borel probability measures (Theorem \ref{DefBeta}).  It is easy to see that  $\gamma(c) <1$ so that $\beta(c) <\log q$ for all $c \in (0,1)$, just because
$$
     \max_x \prod_{j=0}^{q-1}\left|\frac{\sin q \pi (q^jx+c)}{q \sin \pi (q^j x +c)}\right|<1.
$$
A detailed argument is given in \cite{FK2018}.
%But trivially
%$\gamma(0) =1$ or equivalently $\beta(0)=0$.

\medskip

Our main result in this paper is the following theorem concerning the maximal value $\beta(c)$.
\\

\noindent {\bf Main Theorem.} {\em Fix an integer $q\ge 2$. The following hold. \, \ \\
	\indent {\rm  (1)} The supremum in (\ref{Beta}) defining $\beta(c)$
is attained by a unique measure and this measure is $q$-Sturmian. \\
\indent {\rm  (2)} Such a  $q$-Sturmian measure is  periodic
in most cases. More precisely, those parameters $c$ corresponding to non-periodic
Sturmian measures form a set of zero Hausdorff dimension. \\
\indent {\rm  (3)} There is a constant $C>0$ such that
\begin{equation}
	\forall x \in \mathbb{T},  \forall N\ge 1, \ \ \
	 \left|\sum_{n=0}^{N-1} t_{n}^{(q;c)} e^{2\pi i n x}\right|\le C N^{ \gamma(c)}.
\end{equation}
%where $\gamma(c)$ is the best possible.
}
\medskip

A $q$-Sturmian measure is by definition a $T_q$-invariant Borel probability measure
with its support contained in a closed arc of length $\frac{1}{q}$. It is well-known that each closed arc of length $\frac{1}{q}$ supports a unique 
$T$-invariant Borel probability measure. A proof of this fact is included in Appendix A for the reader's convenience.

\iffalse
In the following theorem, we state some properties of
$q$-Sturmian measures, which will be useful to us.
\medskip

\marginpar{I prefer not to state it as a Theorem}
\noindent {\bf Theorem B.} {\em \, \
	Fix $q\ge 2$ and consider the dynamics on $\mathbb{T}$ defined by $Tx=qx\mod 1$. For each $\gamma\in \R$, there is a unique $T$-invariant Borel probability measure $\frak{S}_\gamma$ supported in $C_\gamma=[\gamma,\gamma+q^{-1}] \mod 1 \subset \mathbb{T}$. We have $\frak{S}_\gamma=\frak{S}_{\gamma+1}$ for each $\gamma\in \R$. Moreover, there is a continuous, monotone map $\rho: \R\to \R$ such that\\
	%\begin{itemize}
	\indent {\rm (1)} $\rho(\gamma+1)=\rho(\gamma)+q-1$.\\
	\indent {\rm (2)} $\rho(\gamma)\in \Q$ if and only if $\frak{S}_\gamma$ is supported on a periodic orbit.\\
	\indent {\rm (3)} for each $r\in \Q$, $\rho^{-1}(r)$ is a non-degenerate interval.\\
	\indent {\rm (4)} $\{\gamma\in [0,1): \rho(\gamma)\not\in \Q\}$ has upper Minskowski dimension zero.
	%\end{itemize}
}

\medskip

The results in Theorem B are known. See \cite{Veerman1989}, where more general cases were considered.
For the case $q=2$,  see \cite{BS1994,Bousch2000}.
For related works, see \cite{Boyd1985,Mane1985,Swiatek1988,Veerman1986,Veerman1987}.
%But for reader's convenience,
%we will give a self-contained direct proof in Appendix A.
\fi 

\medskip

For the maximization,
many of the existing results in the literature deal with the case that $f$ is a H\"older continuous function, by Bousch \cite{Bousch2000,Bousch2001}, Jenkinson \cite{J2006,Jenkinson2007,J2008,Jenkinson2009},
Jenkinson and Steel \cite{Jenkinson-S_2010}, Contreras, Lopes and Thieullen  \cite{CLT2001}, Contreras \cite{C2016}, among others. There is a very nice survey paper
\cite{J2017} in which there is a rather complete list of references. See also Anagnostopoulou et al \cite{ADJR2010, ADJR2012b,ADJR2012},  Bochi \cite{Bochi2018}.
%\marginpar{added, to be completed}

\medskip

Up to now, as far as we know, only the exact value of the Gelfond exponent $\gamma(2; 1/2)$ is known, obtained by Gelfond \cite{Gelfond1968}. Some estimate is obtained by Mauduit, Rivat and Sarkozy \cite{MRS2017}.
In Section \ref{sec:compute},  a computer-aided method will be provided to compute the Gelfond
 exponent $\gamma(c)$,  based on the theory
 developed  in Section \ref{sec:max}.   Figure \ref{fig:1} shows the graph of $\gamma(2; c)$
 %. We point out that  this graph is not a complete graph and %that it only shows
 %$\gamma(c)$
 for $c$'s corresponding to periodic Sturmian measures with
 period not exceeding $13$. More %~\marginpar{changed} 
 details can be filled in by using Sturmian measures with larger periods. 
 %It is interesting to 
 Let us %~\marginpar{changed}
  point out that for $c \in (0.428133329021334,0.571866670978666)$, we get the exact value
 \begin{equation}\label{eq: beta_2}
        \beta(2;c) =\log 2+ \frac{1}{2}\log \left|\cos \pi \left(\frac{1}{3}+c\right)
        \cos \pi \left(\frac{2}{3}+c\right)\right|.
 \end{equation}
 The modal around $c=\frac{1}{2}$ of the graph of $\beta(\cdot)$ is nothing
 but the graph of the function on the right hand side of (\ref{eq: beta_2}). This is the contribution of the $2$-cycle
 $\{1/3,2/3\}$. Other details shown in Figure \ref{fig:1} are contributed by other
 cycles.  See (\ref{bc}) for a formula more general than
 (\ref{eq: beta_2}).  The symmetry of the graph of $\gamma(\cdot)$ reflects nothing but the fact $\gamma(q; 1-c)=\gamma(q; c)$ which holds for all
 $c$.

 % However, the above method can not find those
 %$\beta(c)$'s corresponding to non-periodic Sturmian measures.

 \begin{figure}[htb]
 	\centering
 	\includegraphics[width=0.85\linewidth]{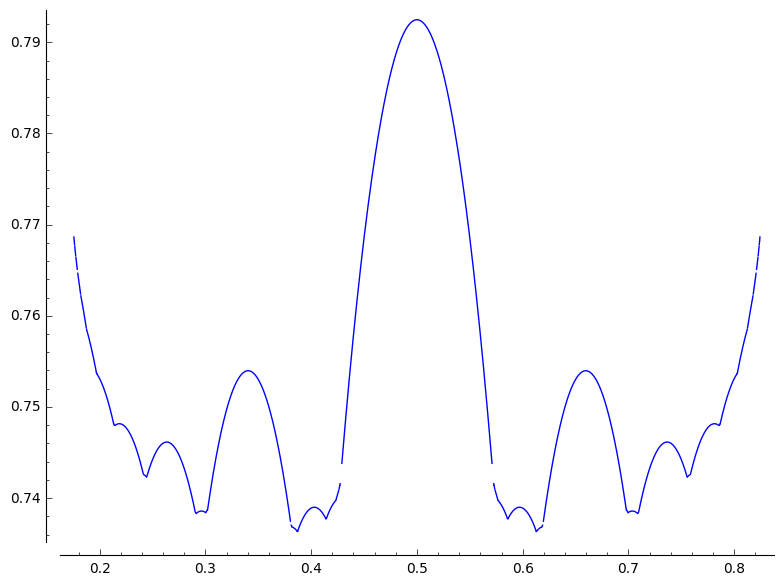}
 	\caption{The graphs of $\gamma(c)$.}
 	\label{fig:1}
 \end{figure}

The Thue-Morse sequence $t_n^{(2;1/2)}=(-1)^{s_2(n)}$ and the  digital sum function $n \mapsto s_2(n)$ are extensively studied in harmonic analysis and number theory
after the works of Mahler \cite{M1927}  and Gelfond \cite{Gelfond1968}.
%Let $\mathcal{A}^+$~\marginpar{Do we need to mention $\mathcal{A}^+$?}
%be 
The set of natural numbers $n$ such that
$s_2(n)$ are even is studied and 
the norms  $\|\sigma_N^{(c)}\|_\infty$ and $\|\sigma_N^{(c)}\|_1$ are involved in the study of the distribution of such sets in $\mathbb{N}$ \cite{Gelfond1968,FM1996a,FM1996b,DT2005}. Queff\'elec \cite{Queffelec2018} showed how to estimate the $L^1$-norm
using the $L^\infty$-norm through an interpolation method.
C. Mauduit and J. Rivat \cite{MR2010} answered a longstanding question of A. O. Gelfond \cite{Gelfond1968} on how the sums of digits of primes are distributed.  This study deals with
$\sum_{p\le N} e^{2\pi i x s_2(p)}$ ($p$ being prime). Polynomials
of the form $\sum_{n\le N} (1+(-1)^{s_2(n)})e^{2\pi i n x}$ are studied
in \cite{AHL2017}. Recently Fan and Konieczny \cite{FK2018} proved that for every $0<c<1$ and  every integer $d\ge 1$ there exist constants $C>0$ and
$0<\gamma_d <1$ such that
$$
    \sup_{ \substack{ q \in \mathbb{R}[x] \\ \deg q \leq d}}
     \left| \sum_{n=0}^{N-1} t_n^{(c)} e^{ 2\pi i q(n)}\right|\le C N^{\gamma_d}.
$$
See also \cite{Konieczny2017}. But the optimal $\gamma_d$ is not known.
\medskip

%\marginpar{The announcement should be rewritten}
A dual quantity is the minimal value
\begin{equation}\label{Alpha}
\alpha(c): = \inf_{\mu \in \mathcal{M}_T} \int_{\mathbb{T}}
   f_c(x) d\mu(x)
%\log |\cos \pi (x+c)| d\mu(x),
\end{equation}
which will play an important role in the study of the pointwise behavior
of $\sigma_N^{(q;c)}(x)$. This minimization and the multifractal analysis of $\sigma_N^{(q;c)}(x)$ are stuided in a forthcoming paper.

We start the paper with  a general setting of dynamical maximization and
minimization (Section 2) and an observation  that the computation
of the Gelfond exponents for generalized Thue-Morse sequences is a
dynamical maximization problem (Section 3). Theorem A will be proved in Section 4 which is the core of the paper.
Section 5 is an appendix, devoted
to the numerical computation of $\beta(c)$ and $\gamma(c)$.

%, even for the Birkhoff limsup. But the liminf
%can take any value less than $-\log 2$. Remark that $\dim %E_0(\alpha(0))=1$.

%For general parameter $c$, the multifractal analysis of %$|\sigma_N^{(c)}(x)|$ seems to be delicate and
%we shall pursue the study  in a forthcoming paper.

\medskip

{\bf Acknowledgements.} The authors are grateful to Thierry Bousch and Oliver Jenkinson for providing useful informations, to Geng Chen for numerical computation and graphic generation.
The first author is supported by NSFC grant no. 11471132 and the
third author is supported by NSFC grant no. 11731003. The first and second authors would like to thank
Knuth and Alice Wallenberg Foundation and  Institut Mittag-Leffler
 (Sweden)
 for their supports.

\section{General setting of maximization and minimization}
 %(The proof follows the published proof of a result of Przytycki, %Rivera-Letelier)

 Let $T: X\to X$ be a continuous map from a compact metric space $X$ to itself. Given an upper semi-continuous function $f: X \to [-\infty, +\infty)$, an interesting and natural problem is {\em ergodic optimization} which asks for %~\marginpar{added} 
 the following maximization
 \begin{equation}\label{Max}
 \beta_f:= \sup_{\mu\in \mathcal{M}_T} \int f(x) d\mu(x)
 \end{equation}
 where $\mathcal{M}_T$ denotes the convex set of all Borel probability $T$-invariant measures.  An {\em $f$-maximizing measure} is by definition
 a probability invariant measure attaining the maximum in (\ref{Max}).

 %Since $f$ is bounded, we can assume that $f$ is negative (i.e taking values in $[-\infty, 0]$).
 What we shall be mostly interested in  is as follows: $X$ is the circle $\mathbb{T}=\mathbb{R}/\mathbb{Z}$, $Tx=qx\mod 1$ for some integer $q\ge 2$, and
 $$ f(x) = \log |\varphi(x)|,$$
 where $\varphi: X\to \R$ is an analytic function not identically zero and moreover,
 $$\varphi''(x)\varphi(x) < \varphi(x)^2$$
 whenever $\varphi(x)\not=0$. That is to say, on any interval where
 $\varphi(x)\not =0$, $\log |\varphi|$ is concave.
 %continuous function, vanishing on a finite set. %not identically zero but may vanish somewhere.
 %{\em like the characteristic function of a closed %set}.~\marginpar{delete?}
 %The most interesting case that we shall deal with is as follows:  $\varphi$
 %is a $1$-periodic function vanishing at one point $b\in [0, 1)]$
 %of class $C^2((b, b+1))$ such that $\log |\varphi|$ is strictly concave in $(b, b+1)$, i.e.
 %$$
 %\forall x \in (b, b+1), \
 %$$
 Such a function $f$ has only  singularities of logarithm type, i.e. if $b$ is a singular point then
 %We shall even assume that $\varphi$ is analytic in $\mathbb{T}$ so that
 %$f$ has a logarithmetic singularity at $b$ i.e
 $$
 \log \varphi (x) \asymp \log |x -b|
 $$
 holds in a neighborhood of $b$.
 A typical example is $\varphi(x)=\frac{\sin \pi q(x+c)}{\sin \pi (x+c)}$ (see Figure \ref{fig:001}).

 \begin{figure}[htb]
 	\centering
 	\includegraphics[width=0.45\linewidth]{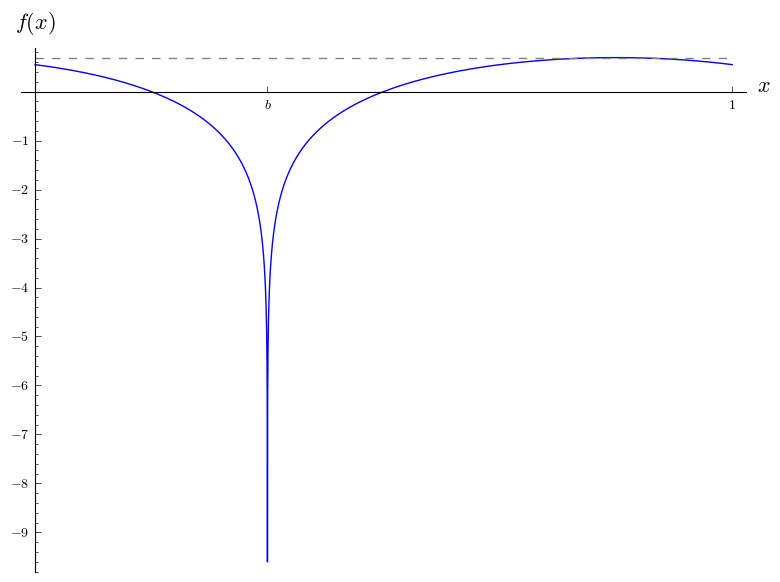}
 	\includegraphics[width=0.45\linewidth]{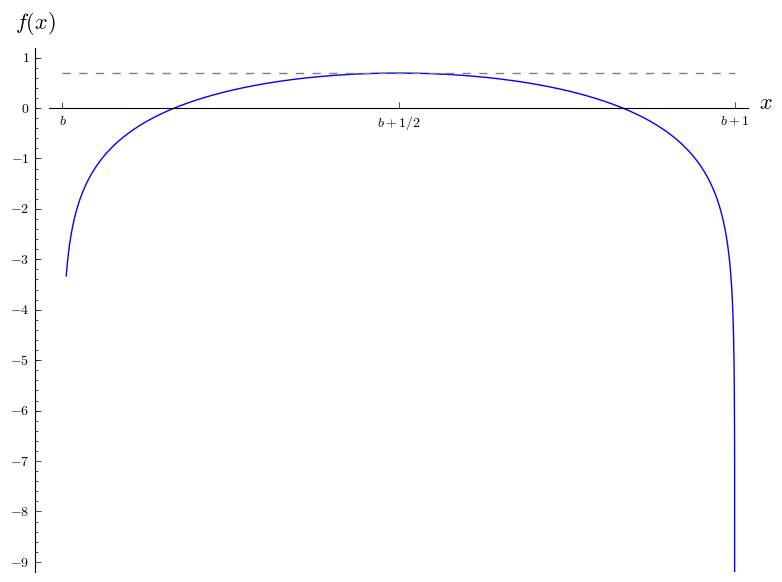}
 	\caption{The graphs of $\log|2\sin \pi (x-b)|$ on the intervals $ [0,1] $ and $ [b,b+1] $ with $b=1/3$.}
 	\label{fig:001}
 \end{figure}
\iffalse
\begin{figure}[htb]
	\centering
	%\includegraphics[width=0.45\linewidth]{01}
	\includegraphics[width=0.65\linewidth]{q=6}
	\caption{The graphs of $\log|\frac{\sin q \pi x}{\sin \pi x}|$ on the interval $ [-q^{-1},1-q^{-1}] $  with $q=6$}
	\label{fig:001}
\end{figure}
\fi

\iffalse
  Another problem is the following minimization
 \begin{equation}\label{Min}
 \alpha_f:= \inf_{\mu\in \mathcal{M}_T} \int f(x) d\mu(x).
 \end{equation}
  An {\em $f$-minimizing measure} is by definition
 a probability invariant measure attaining the minimum in (\ref{Min}).
\fi
%\marginpar{removed}

\medskip

 \subsection{Maximization}
 The points (1) and (2) in the following theorem were proved by Jenkinson \cite{J2006}. They were discussed in \cite{CG} for continuous function $f$.  The point (2) provides  three different ways to describe the maximization (\ref{Max}) through time averages along orbits.
 The point (3) provides a fourth way, using periodic points, in the case of the  dynamics $T_q$.

 Let $R(f)$ be the set of $x$
 such  $\lim_n n^{-1}S_nf(x)$ exists, where
 $$ S_n f(x) = \sum_{k=0}^{n-1}f(T^k x).$$
 %Denote by $\mathcal{P}_T$ the set of periodic measures.
%\marginpar{deleted}

 \begin{theorem} \label{DefBeta} Suppose that $f$ is upper semi-continuous. \\
 	\indent {\rm (1)}
 	The map $\mu \mapsto \int f d\mu$
 is upper semi-continuous so that the supremum in (\ref{Max}) defining $\beta_f$ is attained.\\
 \indent {\rm (2)}
 The maximum value $\beta_f$ is equal to
 \begin{equation*}\label{3Max}
 \sup_{x \in R(f)} \lim_{n\to\infty} \frac{S_n f(x)}{n}
 =\sup_{x \in X} \limsup_{n\to\infty} \frac{S_n f(x)}{n}
 =
 \lim_{n\to\infty}\max_x   \frac{S_n f(x)}{n}.
 \end{equation*}
 \indent {\rm (3)}  Assume $X=\mathbb{T}$, $T(x)=qx\mod 1$,
 and $f(x) = \log |\varphi(x)|$ with $\varphi$ an analytic function having a finite number of zeros. We have
  \begin{equation}\label{Max2}
  \beta_{\log |\varphi|} = \sup_{\mu \in \mathcal{P}_T} \int \log |\varphi| d\mu,
  \end{equation}
  where $\mathcal{P}_T$ denote the collection of all $T$-invariant probability measures supported on periodic orbits.
 \end{theorem}

\begin{proof} (1) and (2) were proved in \cite{J2006}. Here we only
	give an explanation that the last limit in (2) exists.
	Indeed, putting $S_n=\max_{x\in X}S_n f(x)$,  we have $S_{n+m}\le S_n+S_m$,
	%Then
	%P_n(\varphi,x):=\prod_{j=0}^{n-1}|\varphi(T^j(x))|.
	%$$
	%We have $S_nf(x) = \log P_n(\varphi, x)$. Let
	%$$
	%P_n:=P_n(\phi):=\max_{x\in X} P_n(\varphi,x).
	%$$
	%Then $P_{n+m}\le P_n P_m$
	so the  limit exists.
	%~\marginpar{changed}
	
	(3) Let us prove (\ref{Max2}). Obviously the left hand side is not smaller than the right hand side. So it suffices to prove that for any $\mu\in\mathcal{M}_T$ with
	$\int \log |\varphi| d\mu=:\alpha>-\infty$ and any $\eps>0$, there exist a periodic point $p\in \T$ of period $s$ such that
	\begin{equation}\label{eq:goal}
	\sum_{j=0}^{s-1} \log |\varphi(T^j p)|\ge s(\alpha-\eps).
	\end{equation}
	By the ergodic decomposition, we may assume that $\mu$ is ergodic.
	
	We first prove the following claim.\\
	{\bf Claim.}
	{\em Let $\mathcal{C}$ denote the set of zeros of $\varphi$. There exists $\delta_0>0$ such that for $\mu$-a.e. $x\in
	    \T$,    there exists an arbitrarily large positive integer $N$ such that
    }
 $$d(T^j(x),\mathcal{C})\ge q^{-(N-j)}\delta_0\mbox{ for all }0\le j<N.$$

	To prove the claim,  let $\mathcal{C}_\delta :=\{x: d(x,\mathcal{C})<\delta\}$ be the $\delta$-neighbourhood
	of $\mathcal{C}$.
	Since $\log|\varphi|$ is $\mu$-integrable, we must have $\mu(\mathcal{C})=0$ and then $\int_{\mathcal{C}_\delta} \log|\varphi| d\mu\to 0$ as $\delta\to 0$.
	Put $$\text{dep}_\delta(x)=\left\{\begin{array}{ll}
	-\log_q \dist(x, \mathcal{C}) & \mbox{ if } x\in\mathcal{C}_\delta;\\
	0 &\mbox{ otherwise.}
	\end{array}
	\right.$$
	Since $\varphi$ is analytic and non-constant, for any $x_0\in \mathcal{C}$ we have
	%\footnote{We don't really need the analyticity of $\varphi$, but %this local behavior of $\log |\varphi|$.}
	\begin{equation}\label{eq: logf}
	\log |\varphi(x)| = a + m \log d(x, x_0) + o(1)
	\ \ \ {\rm as} \ \ x \to x_0
	\end{equation}
	for some real number $a$ and integer $m\ge 1$.
	Then
	there exist $\delta_0>0$ and $C>0$ such that
	$$
	 \text{dep}_\delta(x)
	\le C|\log\varphi(x)|\quad
	(\forall \delta\in (0,\delta_0],  \forall x\in \mathcal{C}_\delta)
	$$
	Then $\int_{\T} \text{dep}_\delta(x) d\mu\to 0$ as $\delta\to 0$. Choose $\delta_0\in (0,1)$ such that
	$$\int_{\T} \text{dep}_{\delta_0}(x) d\mu <\frac{1}{2}.$$
	Since $\mu$ is ergodic, for $\mu$-a.e. $x\in \T$,
	\begin{equation}\label{eqn:Birkdep}
	\frac{1}{n}\sum_{i=0}^{n-1} \text{dep}_{\delta_0} (T^i x)<\frac{1}{2}, \mbox{ when } n\gg 1.
	\end{equation}
	By Pliss Lemma \cite{Pl}, it follows that there is an arbitrarily large integer $N$ such that for any $0\le j<N$,
	$$\sum_{i=j}^{N-1} \text{dep}_{\delta_0}(T^i(x))< N-j,$$
	and in particular,
\begin{equation}\label{eqn:hyptime}
d(T^j(x), \mathcal{C})\ge \min(q^{-(N-j)},\delta_0)\ge q^{-(N-j)}\delta_0.
\end{equation}
 The claim is proved.
	
		Let us now complete the proof. Fix $\delta_0>0$ as we have chosen above and choose a point $x\in \T$  such that the conclusion of the Claim holds for a sequence of positive integers $N_1<N_2<\cdots$. Choose $x$ suitably so that
	\begin{equation}\label{eqn:Birklogvarphi}
	\frac{1}{n}\sum_{i=0}^{n-1} \log |\varphi(T^i x)|\to \alpha, \mbox{ as } n\to\infty.
	\end{equation}
	Given $\eps>0$, let $\eta>0$ be small such that
	\begin{equation}\label{eqn:smalldist}
	\forall y, y'\in \T, d(y,y')<3\eta d(y,\mathcal{C})\Rightarrow |\log \varphi(y)-\log\varphi(y')|<\frac{\eps}{2}.
	\end{equation}
	Let $z$ be an accumulation point of $\{T^{N_k}(x)\}_{k=1}^\infty$. First fix $k_0$ such that $d(z,T^{N_{k_0}}(x))<\eta\delta_0$.
	Then find $k\gg k_0$ such that $d(T^{N_k}(x),z)<\eta\delta_0$ and
	$$
	     \frac{1}{s}\sum_{j=0}^{s-1} \log |\varphi (T^j(y))|> \alpha-\frac{\eps}{2},
	$$
	where $y=T^{N_{k_0}}(x)$ and $s=N_k-N_{k_0}$. Then
	\begin{equation}\label{eq:faraway}
	d(T^j(y),\mathcal{C})\ge q^{-(s-j)}\delta_0  \quad (\forall  0\le j<s).
	\end{equation}
	We can choose $N_{k_0}$ and $N_k$ such  that $q^{-s}<\eta$.
	
	Let $J:=[y-q^{-s}\delta_0, y+q^{-s}\delta_0]$.
	Since $T^s$ maps J bijectively onto $[T^s(y)-\delta_0,T^s(y)+\delta_0]\supset J$, there exists $p\in J$ such that $T^s(p)=p$.
	%Moreover, $d(T^s(y),p)=2^s d(y,p)$, so $$d(y,p)=d(T^s(y), %p)/(2^s-1)< 2\eta \delta_0/(2^s-1).$$
	 Notice that  for  $0\le j<s$,
	$$
	d(T^j(y), T^j(p))=q^{-(s-j)} d(T^sy , p)<q^{-s+j}\cdot 3\eta \delta_0 < 3\eta d(T^j(y),\mathcal{C}),
	$$
	because $d(T^s y,z)<\eta \delta_0$, $d(z, y)<\eta \delta_0$
	and $d(y, p) \le q^{-s}\delta_0<\eta\delta_0$.
	For the last inequality we used (\ref{eq:faraway}).
	According to (\ref{eqn:smalldist}), this
	implies that
	$$\log \varphi (T^j(p))-\log\varphi(T^j(y))>-\frac{\eps}{2}.
	\quad (\forall 0\le j <s).$$
	Therefore
	$$\frac{1}{s}\sum_{j=0}^{s-1} \log \varphi (T^j(p))>\frac{1}{s}\sum_{j=0}^{s-1}\log\varphi(T^j(y))-\frac{\eps}{2}
	>\alpha-\eps.
	$$
	Thus (\ref{eq:goal}) is proved.	
\end{proof}

 \section{Gelfond exponent and maximization problem}

 We approach the computation of Gelfond exponent from the point of ergodic optimization. Throughout we fix an integer $q\ge 2$ and will drop the superscript $q$ from notation. Recall that  $T$ denotes the map $x\mapsto qx \mod 1$ on the circle $\T=\R/\Z$. For each $c\in \R$, put
 $$f_c(x):=\log \varphi_c(x)  \ {\rm with }\
 %\varphi_c(x)=1+\cos 2\pi (x+ c).
    \varphi_c(x)=\left|\frac{\sin \pi q(x+ c)}{ \sin \pi (x+c)}\right|.
 $$
 So $\varphi_c(x)=\varphi_0(x+c)$ and $f_c(x)=f_0(x+c)$.
% We shall denote
 %$$
 %\beta(c)=\sup_{\mu\in\mathcal{M}_T} \int_\mathbb{T} f_c d\mu.
% $$

 Fix $x\in\R$ and  consider the function $w_x^{(c)}:\N\to \C$ defined by
 $$w_x^{(c)}(n): = t_n^{(c)} e^{2\pi i n x} =e^{2\pi i(cS_q(n)+nx)},$$
 which is   $q$-multiplicative in the sense that
 $$
 w_x^{(c)} (a q^t + b)
 = w_x^{(c)} (a q^t) w_x^{(c)} (b)
 $$
 for all non-negative integers $a, b$
 and $t$ such that $b<q^t$ (see \cite{Gelfond1968}).
 Using this multiplicativity we can
 establish a relationship between Gelfond exponents and dynamical
 maximizations.

 \subsection{Gelfond exponent and maximization}
  Indeed, the $q$-multiplicativity gives rise to
\begin{equation*}%
        \sigma_{q^n}^{(c)}(x) = \prod_{k=0}^{n-1} \sum_{j=0}^{q-1} w_x^{(c)}(j\cdot q^k).
        %|w_x^{(c)}(0) + w_x^{(c)}(q^k)|=q^n \prod_{j=0}^{n-1}|\cos \pi (2^j x +c)|.
\end{equation*}
Since the above sum is equal to
$$%\sum_{j=0}^{q-1} w_x^{(c)}(j\cdotp q^k)=
\sum_{j=0}^{q-1} e^{2\pi i j (c+q^k x)}=\frac{1-e^{2\pi i q(c+q^k x)}}{1-e^{2\pi i(c+q^k x)}}
=\frac{e^{\pi i q(c+q^k x)}}{e^{\pi i (c+q^k x)}} \frac{\sin \pi q(q^k x+c)}{\sin \pi (q^k x+c)},$$
we get % this gives us
\begin{equation}\label{S-P}
\left|\sigma_{q^n}^{(c)}(x)\right|
%=\prod_{k=0}^{n-1}   \frac{\sin \pi q(q^k x+c)}{\sin \pi (q^k x+c)}
= \prod_{k=0}^{n-1} \varphi_c (q^k x).
\end{equation}
Therefore, (\ref{eqn:BN})  is equivalent to the following estimation:
\begin{equation}\label{eqn:BN'}
\sup_{x\in \mathbb{R}}\prod_{j=0}^{n-1}\varphi_c(q^j x)=O(q^{n\gamma}).
\end{equation}
In particular, $\gamma(c)$ is also the infimum of $\gamma$ for which (\ref{eqn:BN'}) holds. The function $\gamma(\cdot)$ has the following symmetry.

\begin{prop}\label{symmetry} We have $\gamma(c) = \gamma(1-c)$ for all $c \in [0,1]$.  Moreover, for all $n \ge 1$ and all $x \in [0,1]$ we have
$$
       \prod_{j=0}^{n-1} \varphi_c(q^j x) =   \prod_{j=0}^{n-1} \varphi_{1-c}(q^j (1-x)).
$$
\end{prop}

\begin{proof} This follows simply from the parity and the $1$-periodicity of $\varphi(x):=\varphi_0(x)$ which gives
$$
      \varphi(x+c) = \varphi (-x - c) = \varphi (-x  + (1-c)),
$$
and of the fact  $-q^n x = q^n(1-x) \mod 1$.
%Thus
%$$
%     \max_x \prod_{k=0}^{n-1} \varphi_{c}(2^k x) = \max_x %\prod_{k=0}^{n-1} \varphi_{1-c}(2^k x).
%$$
\end{proof}

By definition, the sequences  $(t_n^{(2;3/4)})$ and $(t_n^{(2;1/4)})$  are related in the following way
$$
      \frac{t_n^{(2;3/4)}}{t_n^{(2;1/4)}} = t_n^{(2;1/2)}=(-1)^{S_2(n)}.
$$
An amazing relation!  Apparently, $(t_n^{(2;3/4)})$ and $(t_n^{(2;1/4)})$ seem very different, but  $|\sigma_{2^n}^{(2;3/4)}(x)| = |\sigma_{2^n}^{(2;1/4)}(1-x)|$.
%This seems not clear from the initial definition by (\ref{GTM}).
%\marginpar{I do not agree with the last sentence.}

\begin{prop}\label{prop:b-g} We have  $\gamma(c)= \frac{\beta(c)}{\log q}$ for each $c$.
\end{prop}

This is a consequence of (\ref{eqn:BN'}) and Theorem \ref{3Max}.

\section{Maximization for $f_c$ and Sturmian measures}\label{sec:max}
In this section, we consider the maximizing problem in our most interesting particular case.
 Let $T$ denote the map $x\mapsto qx$ on the circle $\T=\R/\Z$, and for each $c\in \R$, put
$$f_c(x):=\log \varphi_c(x), \ \ \ {\rm with} \ \  \varphi_c(x)= \left|\frac{\sin \pi q (x+c)}{\sin \pi (x+c)}\right|.$$
%Let $\mathcal{M}_T$ denote the collection of all $T$-invariant Borel %probability measures.
Recall that our object of study is to find
$$\beta(c)=\sup_{\mu\in\mathcal{M}_T} \int_\mathbb{T} f_c d\mu.$$

For each $\lambda\in \R$, there is a unique $T$-invariant measure $\frak{S}_\lambda$ that is supported in circle arc $C_\lambda: =[\lambda, \lambda+q^{-1}]$ $\mod 1$, called $q$-Sturmian measure. These measures $\frak{S}_\lambda$ are ergodic and $\frak{S}_\lambda=\frak{S}_{\lambda'}$ whenever $\lambda'-\lambda\in\mathbb{Z}$. See Appendix A for a proof of these facts. %~\marginpar{added}
%In the case $q=2$, these measures are known as {\em Sturmian %measures}. We shall call such a measure a {\em $q$-Sturmian %measure}.

The main result of this section is the following theorem.

\begin{theorem}\label{thm:main} Fix an integer $q\ge 2$.%~\marginpar{added} 
	For any $c\in \R$, $f_c$ has a unique maximizing measure $\nu_c$. The measure $\nu_c$ is a $q$-Sturmian measure. Moreover, there exists a constant $C>0$, which is independent of $c$, such that
$$\sum_{i=0}^{n-1} f_c(q^n x)-n \beta(c)\le C,$$
for each $x\in \mathbb{R}$ and each $n\ge 1$.
\end{theorem}

To prove Theorem \ref{thm:main}, we shall apply and extend the theory of Bousch-Jenkinson. An important fact that is used in the argument is that $f_c$ is strictly concave away from the singularties, or equivalently that is the same for $f_0$:
$$f_0''(x)= \pi^2\left(\frac{1}{\sin^2 \pi x}-\frac{q^2}{\sin^2 \pi qx}\right)<0.$$

\iffalse
Theorem \ref{thm:main} remain true for any  $f$ which has  a logarithmetic singularity at $b$, but is elsewhere real analytic and strictly concave on $(b, b+1)$.
That is the case for $f_c$ with $b(c) = \frac{1}{2}-c$ because
$$
    f_c''(x)=-\pi^2\sec^2(\pi(x+c))<0.
$$

Such concavity enables us to apply the method of Jenkinson \cite{J2008}, which was motivated by Bousch \cite{Bousch2000}. Our main goal in this section is to prove
Theorem \ref{thm:main}. Our proof is presented in the case of $f_c$
but there is no significant differences for the above mentioned general case.
\fi

We shall first recall the pre-Sturmian and Sturmian condition introduced by Bousch~\cite{Bousch2000}. Bousch introduced these concepts in the case $q=2$ which extends to the general case in a straightforward way.

\subsection{Pre-Sturmian condition and Sturmian condition}

%In order to prove Theorem \ref{thm:main}, we need
%to recall the pre-Sturmian and Sturmian conditions introduced by Bousch
%\cite{Bousch2000}.

 For each $\gamma\in \R$, let $$C_\gamma:=[\gamma,\gamma+{1}/{q}]\mod 1\subset \T$$
 be the arc in $\T$, starting from $\gamma$ and rotating in the anti-clockwise direction.
 %Let $\frak{S}_\gamma$ denote the unique $T$-invariant measure supported in $C_\gamma$,  called {\em $q$-Sturmian measure} or simply Sturmian %measure, see Theorem B %which does exit and is unique
 %(see also \cite{BS1994} for the case $q=2$).
% \marginpar{deleted}
  Let $C_\gamma'=[\gamma, \gamma+1/q)\mod 1\subset\T$ and
$\tau_\gamma: \T\to C_\gamma'$ denote the inverse branch of $T$ restricted on $C_\gamma'$. So $\tau (Tx)$ is the unique point in $C_\gamma'$ such that $q(\tau(Tx)-x)\in \Z$.
%Let us remark that
%\begin{equation}\label{eqn:tauT}
%\tau_\gamma (T(x))=x+\frac{1}{2}, \ \ \ \forall x\in \T^1\setminus C_\gamma'.
%\end{equation}
%Indeed,  since $T^{-1}(Tx) = \{x, x +1/2\}$, $\tau_\gamma(T(x))$ is the only preimage of $T(x)$ that is not equal to $x\mod 1$.
%~\marginpar{changed for precision}
\medskip

 The following definition comes from Bousch \cite{Bousch2000} which discusses  the case $q=2$ with $f$ supposed Lipschitzian.
 We will only assume that $f$ is Lipschtzian on $C_\lambda$.

\begin{definition} Let $f: \mathbb{T} \to [-\infty, +\infty)$ be a Borel function and let $\lambda\in\R$. 
%bounded from above such that $f$ is Lipschtzian on $C_\lambda$. 
We say that $f$ satisfies the {\em pre-q-Sturmian condition} for $\lambda$, if $f$ is Lipschitz on $C_\lambda$ and there exists a Lipschitz function $\psi:\T\to \R$ and a constant $\beta\in \mathbb{R}$ such that
	\begin{equation}\label{pre-sturm}
	\forall x\in C_\lambda, \  f(x)+\psi(x)-\psi\circ T(x)=\beta.
	\end{equation}
	 If, furthermore,
	\begin{equation}\label{sturm}
	\forall y\in \T\setminus {C}_\lambda, \  f(y)+\psi(y)-\psi\circ T(y)<\beta,
		\end{equation} then we say that $f$ satisfies the {\em q-Sturmian condition} for $\lambda$.
		%~\marginpar{changed}
\end{definition}

%Notice that $C_\gamma$ contains no singularity of $f_c$.
\iffalse
By the continuity of $f$ on $C_\lambda$, (\ref{pre-sturm}) holds if the equality in (\ref{pre-sturm}) holds for  $x \in C_\lambda'$.
Assume that $f$ satisfies the pre-Sturmian condition for $\lambda$, then\\
\indent (1) \ $\beta = \int f d \frak{S}_\lambda$;\\
\indent (2) \ $\psi$ is unique up to an additive constant and actually
\begin{equation}\label{psi'}
a.e. \ \ \ \psi'(x) = \sum_{n=1}^\infty \frac{f_c'(\tau_\lambda^n x)}{q^n}.
\end{equation}
\fi 
%\marginpar{removed}

\iffalse
As Bousch showed, the key for proving results like those in Theorem \ref{thm:main} is to check that
$f_c$ satisfies the Sturmian condition for some $\gamma$. As we shall see,
a candidate function $\psi$ in the definition of the pre-Sturmian condition is
the Lipschitz function $\psi$ with derivative
$$
      \psi'(x) = \sum_{n=1}^\infty \frac{f_c'(\tau_\gamma^n x)}{q^n}.
$$
\fi

To study the pre-Sturmian condition, let us consider the first time to leave $C_\gamma'$
\begin{equation}\label{e_gamma}
e_\gamma(x):=\inf\{k\ge 0: T^kx\in \T\setminus {C}_\gamma'\} =\sum_{n=0}^\infty \chi_{\tau_\gamma^n(C_\gamma')}(x).
\end{equation}
Let $A_1=C_\gamma'$ and $A_n = C_\gamma' \cap T^{-1}A_{n-1}$ for $n>1$.
Then
$$
       A_n
     %= C_{\gamma} \cap T^{-1}C_\gamma \cap \cdots \cap %T^{-(n-1)}C_\gamma
       = C_{\gamma}' \cap \tau_\gamma(C_\gamma') \cap \cdots \cap \tau_\gamma^{n-1}(C_\gamma')
            =\tau_\gamma^{n-1}(C_\gamma').
$$
From this we verify the second equality in (\ref{e_gamma}). Thus $e_\gamma \in L^1$ and
$$\int_{\mathbb{T}} e_\gamma (x) d x=\sum_{n=1}^\infty q^{-n}=\frac{1}{q-1}.
$$
Since $\tau_\gamma (C_\gamma') \subset C_\gamma'$, the function $e_\gamma$
is supported by $C_\gamma'$.

We have the following criterion for the pre-Sturmian condition, due to Bousch \cite{Bousch2000} (p.503).
%\marginpar{need to check}
\begin{prop}%[Bousch \cite{Bousch2000}, p. 497, p.503]
	\label{prop:bouschsturm}\
	 Let $f: \mathbb{T} \to [-\infty, +\infty)$ be a Borel function bounded from above.
% such that $f$ is Lipschtzian on $C_\lambda$.
\begin{enumerate}
\item If $f$ satisfies the q-Sturmian condition on $C_\lambda$ for some $\lambda\in\R$, then $\frak{S}_\lambda$ is the unique maximizing measure of $f$.
%\item There is at most one $\gamma\in (-q^{-1}-c, -c)$ such that %$f_c$ satisfies the q-Sturmian condition for $\gamma$.
\item $f$ satisfies the pre-q-Sturmian condition for $\lambda$ if and only if $f$ is Lipschitzian on $C_\lambda$ and 
$$v_f(\lambda):=\int_{C_\lambda} f'(x) e_\lambda(x)dx=0.$$
\end{enumerate}
\end{prop}

\begin{proof} These results were stated in \cite{Bousch2000}
	for Lipschitzian $f$. But only the Lipschtzian condition
	on $C_\lambda$ is actually needed. We repeat here the main lines of proofs for the convenience of reading.
	
	(1) It is clear that $\beta$ is attained by the Sturmian measure. On the other hand, any other invariant measure $\mu$ has a support intersecting $\mathbb{T}\setminus C_\lambda$, by the uniqueness of Sturmian measure supported by $C_\lambda$. Then $\int f d \mu <\beta$ by the  Sturmian condition.
	
	(2) Let $\tau=\tau_{\lambda}$. Assume the pre-Sturmain condition which can be restated as
	$$
	    \forall x \in \mathbb{T}\setminus \{q\lambda\},
	    \quad \psi (x) = -\beta + \psi(\tau x) + f(\tau x).
	$$
	By differentiating  and iterating, we get
	$$
	     a.e. \ \    \forall N\ge 1, \quad\psi'(x) =\sum_{n=1}^N \frac{f'(\tau^n x)}{q^n} + \frac{\psi'(\tau^n x)}{q^N}.
	$$
	Since $\psi$ is Lipschitzian, $\psi'$ exists almost everywhere  and $\psi' \in L^\infty(\mathbb{T})$.
	Letting $N\to \infty$, we get the following formula
    \begin{equation}\label{psi'}
    a.e. \ \ \ \psi'(x) = \sum_{n=1}^\infty \frac{f'(\tau_\lambda^n x)}{q^n}.
    \end{equation}
    Then integrate it to obtain
	\begin{equation}\label{periodic}
	     0 = \psi(1) - \psi(0) =\int_0^1 \psi'(x) dx = \int_{C_\lambda} f'(x) e_\lambda(x) dx.
	\end{equation}
	Now assume $\int_{C_\lambda}f'(x) e_\lambda (x) dx =0$ and 
	$f$ is Lipschitzian on $C_\lambda$. Then $f'$ exists almost everywhere on $C_\lambda$  and $f' \in L^\infty(C_\lambda)$.
	Since $f' \in L^\infty(C_\lambda)$,  the series in (\ref{psi'}) defines a bounded function $\psi'$ then
	a Lipschitzian function $\psi$. The computation (\ref{periodic}) shows that $\psi$ is $1$-periodic.
	The formula (\ref{psi'})
	can be rewritten as
	$$
	   a.e. \quad  \psi'(x)  -\frac{1}{q} \Big(  f'(\tau x) + \psi'(\tau x)\Big) =0.
	$$
	In other words, the Lipschitzian function
	$\psi (x) - f(\tau x) - \psi(\tau x)$ is a constant, say $-\beta$.
	\end{proof}

In the case $f=f_c$, we will first prove that the pre-Sturmian condition is satisfied and then prove that the pre-Sturmian condition implies the Sturmian condition.  So, by
Proposition \ref{prop:bouschsturm}, the maximizing measure of $f_c$ is unique and it is a Sturmian measure.
\medskip

%the we are going to show
%It is usually easy to check the pre-Sturmian condition. In our %setting, we have a very precise description as follows.
For any $c \in \R$, we are going to look for $\lambda \in (-q^{-1} -c, -c)$ such that $f_c$ satisfies the pre-Sturmian condition on $C_\lambda$, i.e. $v_c(\lambda)=0$. But $c \mapsto v_c(\lambda)$ can be considered as a $1$-periodic function on $\R$. So,
set  $$\Omega=\{(c,\lambda) \in \R^2: \lambda \in (-q^{-1}-c, -c)\}$$ and $$\Omega_0=\{(c,\lambda)\in \Omega: v_c(\lambda)=0\}.$$
The following lemma shows that the equation $v_c(\lambda) =0$ does have a
real solution $\lambda$ for every real $c$, so that $f_c$ satisfies the pre-Sturmian condition for any $c$. Actually for every fixed $\lambda$, it  will be proved that there exists a unique number $\textbf{c}(\lambda)$
such that $(\textbf{c}(\lambda), \lambda) \in \Omega_0$ and that $\lambda \mapsto \textbf{c}(\lambda)$ is an almost Lipschitzian homeomorphism from $\R$ onto $\R$.

\begin{lemma}\label{lem:vcgamma} There is a homeomorphism $\textbf{c}: \R\to \R$ such that
\begin{equation}\label{eqn:graphq}
\Omega_0=\{(\textbf{c}(\lambda), \lambda):\lambda\in \R\}.
\end{equation}
Moreover, 
\begin{enumerate}
\item the function $\textbf{c}(\lambda)$ has modulus of continuity $O(|x\log x|)$.
\item there exists $\eps>0$ such that $$-\frac{1}{q}+\eps\le \textbf{c}(\lambda)+\lambda\le -\eps.$$
\end{enumerate}
\end{lemma}
\begin{proof}
	%Recall that
	%$f_c(x) =  \frac{1}{2} \log \cos^2 \pi (x +c)$ and
	%$$
	%f_c'(x) =  - \pi \tan \pi (x+c), \quad
	%f''_c(x) = - \frac{ \pi^2}{\cos^2 \pi (x+c)} \ \ \  (x \not= \frac{1}{2} -c).
For each fixed $\lambda \in \R$, the function $c\mapsto v_c(\lambda)$ is clearly smooth on $(-q^{-1}-\lambda, -\lambda)$ and
$$\frac{\partial v_c(\lambda)}{\partial c}
= \int_{\mathbb{T}} f''_c(x) e_\lambda(x) dx
=\sum_{n\ge 0} \int_{\tau_\lambda^n(C_\lambda)} f''_c(x) dx\le -\frac{K}{q-1}<0,$$
where $-K=\max \{f_0'(x): x\in (-q^{-1}, q^{-1})\}<0$.
Thus for each $\lambda$, there is at most one $c$ with $(c,\lambda)\in \Omega_0$.
%by Proposition \ref{prop:bouschsturm} (3).
%~\marginpar{deleted}
On the other hand,
observe that for each $c\in\R$, $f_c'(x)>0$ for all $x\in (-q^{-1}-c, -c)$ and $f_c'(x)<0$ for all $x\in (-c, q^{-1}-c)$.
See Figure \ref{fig:f'} for the graph of $f_0'$ and the graph of 
$f_c'$ is nothing but a translation of that of $f_0'$. 
 \begin{figure}[htb]
	\centering
	\includegraphics[width=0.80\linewidth]{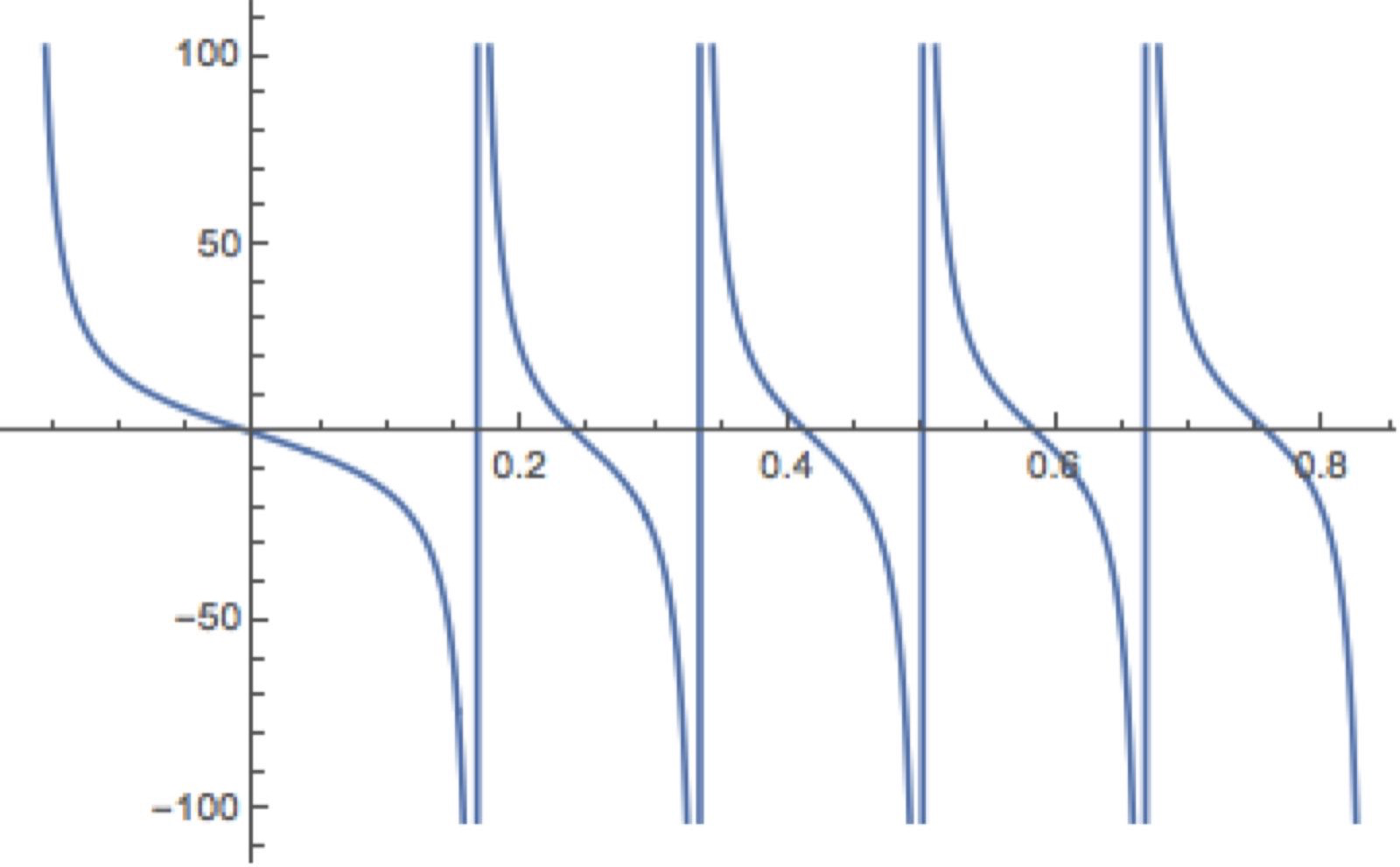}
	\caption{The graphs of $f'_0$ on the interval % $ [0,1] $ and
		$ [-1/q,1-1/q] $, here $q=6$.}
	\label{fig:f'}
\end{figure}
 
As $\lambda\searrow -q^{-1}-c+0$, $C_\lambda$ tends to $[-q^{-1}-c, -c]$. Since $e_\lambda$ is supported in $C_\lambda$, this implies that
\begin{equation}\label{eqn:lim1}
\lim_{\lambda\searrow -q^{-1}-c} v_c(\lambda)>0.
\end{equation}
Similarly we show that
 \begin{equation}\label{eqn:lim2}
 \lim_{\lambda\nearrow q^{-1}-c} v_c(\lambda)<0.
\end{equation}
% or $c\to -\gamma-0$, the segment $[-q^{-1}-c, -c]$ tends to $C_\gamma$.
%As $c\to 1-\gamma-0$, the semi-circle $[1-c, 3/2-c]$ tends to %$C_\gamma$ too.
%Also observe that
%Since $e_\gamma$ is supported by $C_\gamma$,
%it follows that
%$\begin{equation}\label{2limits}
%	\lim_{c\to 1/2-\gamma+0} v_c(\gamma)\in (0,\infty]\mbox{ and } \lim_{c\to 1-\gamma-0} v_c(\gamma)\in [-\infty, 0).
%\end{equation}
By the Intermediate Value Theorem, for each $c$, there is one $\lambda\in\R$ with $(c,\lambda)\in \Omega_0$.
A similar argument by the Intermediate Value Theorem shows that for each $\lambda$, there is $c\in\R$ with $(c,\lambda)\in\Omega_0$. It follows that there is a bijective function $\textbf{c}:\R\to \R$ such that (\ref{eqn:graphq}) holds.
%thus $q$ is surjective.
By \cite{Bousch2000} (p. 505),
$\lambda\mapsto e_{\lambda}$, as a function from $\R$ to $L^1(\R/\Z)$, has modulus of continuity $O(|x\log x|)$. This, together with the uniform upper bounds on $\frac{\partial v_c}{\partial c}$,  implies that $\textbf{c}$ is continuous with modulus of continuity $O(|x\log x|)$. In particular, $\textbf{c}:\R\to\R$ is a homeomorphism.

Finally the statement (2) holds because $\lambda\mapsto \textbf{c}(\lambda)+\lambda$ is of period $1$ and it takes values in $(-1/q, 0)$. 
\end{proof}
%\marginpar{changed $q$ to $\textbf{c}$}
\subsection{Pre-Sturmian condition implies Sturmian condition for $f_c$}
Bousch mentioned that in the case $q=2$, the pre-$q$-Sturmian condition, in practice, often implies the stronger $q$-Sturmian condition. Jenkinson noticed that it is {\em always} the case for continuous maps $f:\T\to \R$ which is strictly concave on $(0,1)$. We shall develop further Jenkinson's argument to show the following:
%see that it is still the case for any function $f$ which has a singularity $b$ and is concave in the open interval $(b, b+1)$.

\begin{prop}\label{lem:Jenkinsonlemma} If $f_c$ satisfies the pre-q-Sturmian condition for some $\lambda\in (-q^{-1}-c, -c)$, then
$f_c$ satisfies the Sturmian condition for $\lambda$.
\end{prop}

The proof of this proposition is complicated and will be postponed to the next section.

\subsection{Proof of Theorem~\ref{thm:main}}
By Lemma~\ref{lem:vcgamma} above, $f_c$ satisfies the pre-$q$-Sturmian condition for some $\lambda\in [-q^{-1}-c+\eps, -c-\eps]$. By Proposition~\ref{lem:Jenkinsonlemma}, $f_c$ satisfies the $q$-Sturmian condition for this $\lambda$. Thus there is a Lipschitz function $\psi: \T\to \R$ and a constant $\beta$ such that
$$F(x):=f_c(x)+\psi(x)-\psi(T(x))=\beta, \forall x\in C_\gamma,$$
and $$F(x)<\beta, \forall x\in \T\setminus C_\gamma.$$
Moreover, by (\ref{psi'}), there exists $C$ depending only on $\eps$ such that $\|\psi'\|_\infty\le C.$
By Proposition~\ref{prop:bouschsturm}, the Sturmian measure $\frak{S}_\gamma$ is the unique maximizing measure of $f_c$, $\beta=\beta(c).$ Clearly, for all $x\in \T$,
\begin{multline*}
S_n f(x)-n\beta(c)= S_n F(x) -\psi(x)+\psi(T^n(x))-n\beta\\
\le \psi(T^n(x))-\psi(x)\le \|\psi'\|_\infty\le C.
\end{multline*}
%is bounded from above by a constant.

\subsection{$\nu_c$ is periodic for almost all $c$} %~\marginpar{changed}
Recall that $\nu_c$ denotes the maximizing measure of $f_c$
(see Theorem \ref{thm:main}).
Let $$\mathcal{P}=\{c\in \R: \nu_c \ \mbox{ is NOT supported on a periodic orbit}\}\\
.$$
\begin{theorem}\label{thm:typicalper} The set $\overline{\mathcal{P}}$ is nowhere dense and has Hausdorff dimension zero.
%The following assertions hold:\\
%	%\indent (1)\ The Hausdorff~\marginpar{changed} dimension of $\Gamma\cap [0,1)$ is zero;\\ %~\marginpar{changed}
%	\indent (1)\ The Hausdorff~\marginpar{changed} dimension of $\mathbb{R} \setminus \mathcal{P}$ is zero;\\
%	\indent (2)\ The set $\mathcal{P}$ contains an open and dense subset of $\R$.  %~\marginpar{changed}
\end{theorem}

\begin{proof} Let $\textbf{c}$ be the function as in Lemma~\ref{lem:vcgamma} and let 
$$\Gamma=\{\lambda\in \R: \frak{S}_\lambda \text { is NOT supprted on a perioic orbit}\}.$$
Then $\mathcal{P}=\textbf{c}(\Gamma)$. 
Since $\textbf{c}$ is a homeomorphism, $\overline{\mathcal{P}}=\textbf{c}(\overline{\Gamma})$. 
By Proposition~\ref{prop:qStur}, $\overline{\Gamma}$ has Hausdorff dimension zero. %
% $$\mathcal{P}':=\textbf{c}^{-1}(\mathcal{P})=$$
% 
%For each $\lambda\in \R$, let $R_\lambda:\T\to\T$ denote the continuous map such that $R_\lambda|C_\lambda= T|C_\lambda$ and such that $R_\lambda$ %is constant on $\T\setminus C_\lambda$. This is a continuous monotone map from $\T$ to itself and has a well-defined rotation number %$\rho(\lambda)\in \T$. The map $\lambda\mapsto \rho(\lambda)$ is continuous and monotone. Let
%$$\Gamma=\{\lambda\in \R: \rho(\lambda)\not\in\Q/\Z\}.$$ By \cite{Veerman1989}, $\dim(\Gamma)=0$.
%It is well-known that $\Gamma=\R\setminus \mathcal{P}'$. 
Since $\textbf{c}$ is $\alpha$-H\"older for each $\alpha\in (0,1)$, it follows that 
$$\dim(\overline{\mathcal{P}})=\dim (\textbf{c}(\Gamma))=0,$$
which also implies that $\mathcal{P}$ is nowhere dense. 
%
%
%It is clear that
% $\mathcal{P}'\setminus \text{int}\mathcal{P}'$ is a countable set since for any $\lambda$ in this set, $\lambda\mod 1$ is a periodic point of $T$. %Thus $\R\setminus \text{int}(\mathcal{P}')$ also has Hausdorff dimension zero. In particular, $\text{int}(\mathcal{P}')$ is an open dense subset of
% $\R$. Since $\textbf{c}$ is a homeomorphism, the statement (2) follows. %$$\overline{\Gamma
\end{proof}

\begin{remark} We learned from Bousch (personal communication) that  any bounded subset of $\Gamma$ has upper Minkowski dimension $0$, and hence so does any bounded subset of $\mathcal{P}$.
\end{remark}

\section{Pre-Sturmian condition implies Sturmian condition}
\label{sect:PStoS}
\iffalse
Given an integer $q\ge 2$, write
$$f(x)=\log\left|\frac{\sin (\pi qx)}{q\sin (\pi x)}\right|$$
and for each $c\in \R$, $f_c(x)= f(x+c)$.
As a function defined on $\T$, the function $f_c$ has $q-1$ singularities $j/q -c$, $1\le j<q$ and outside these singularities,
$f_c$ is concave. In fact,
$$f'(x)=\pi \left(q\cot \pi qx -\cot \pi x\right)$$ and
$$f''(x)=\pi^2 \left(\frac{1}{\sin^2 \pi x} -\frac{q^2}{\sin^2 \pi qx}\right)<0.$$
Let $T=T_q:\T\to\T$ denote the map $x\mapsto qx \mod 1$.
\fi
The goal of this section is to prove Proposition~\ref{lem:Jenkinsonlemma} which we restate as

\begin{theorem}\label{Pre-S} Assume that $f_c$ satisfies the pre-Sturmian condition on $C_\lambda=[\lambda, \lambda+1/q]$ for some $\lambda\in (-1/q-c, -c)$. Then $f_c$ satisfies the Sturmian condition on $C_\lambda$.
\end{theorem}

The pre-Sturmian condition says that there exists  Lipschitz function $\psi: \T\to\R$  such that $$F(x): =f_c(x)+\psi(x)-\psi(Tx)$$
 is constant (denoted by $\beta$) on $C_\lambda$. Let $\tau: \T\to [\lambda, \lambda+1/q)=:C'_\lambda$ denote the inverse branch of $T$.  By Proposition~\ref{prop:bouschsturm} and its proof,
\begin{equation}\label{eqn:varphi'}
\psi'(x)=\sum_{n\ge 1}\frac{f_c'(\tau^n(x))}{q^n}, a.e.
\end{equation}
and
\begin{equation}\label{eqn:preStur}
\sum_{n=1}^\infty \int_{\tau^{n-1} (C_\lambda)} f_c'(x) dx=0.
\end{equation}

Proving Theorem \ref{Pre-S} is to check
$F(x) <\beta$ for $x$ outside $C_\lambda$. 
Before going to details which are unfortunately quite cumbersome, let us describe the strategy. It suffices to show that $F(x)< F(y)$ for some $y\in C_\lambda$. Put $f=f_0$. Then 
$$F(x)-F(y)=f(x+c)-f(y+c) + \psi(x)-\psi(y) -(\psi(Tx)-\psi(Ty)).$$
The estimate on $f(x+c)-f(y+c)$ will be based on the formula defining $f$, which is often a negative number with `big' absolute value and contributes as the `main term'. An upper bound on $\psi(x)-\psi(y)$ can be deduced from the formula (\ref{eqn:varphi'}). An lower bound on $\psi(T(x))-\psi(T(y))$ can also be deduced from (\ref{eqn:varphi'}), although we shall often use simply the fact $\psi(T(x))=\psi(T(y))$ if $q(x-y)\in \Z$.
\medskip

We will have to distinguish three cases according to the location of $x$. First let us present $\mathbb{T}\setminus C_\lambda$ as follows
$$ \mathbb{T}\setminus C_\lambda= J^-\cup J^+,$$
where $$J^-:=\Big(\lambda -\frac{q-1}{2q}, \lambda\Big),\quad J^+:= \Big(\lambda +\frac{1}{q}, \lambda +\frac{q+1}{2q}\Big].
$$
Let also $$ M=\Big(-c-\frac{1}{q}, -c+\frac{1}{q}\Big), \quad
   C^{-} = \Big(-c -\frac{1}{q}, \lambda\Big), \quad  C^+=\Big(\lambda +\frac{1}{q},  -c+ \frac{1}{q}\Big).
$$
So, $M$ is the disjoint union of $C^-, C_\lambda$ and $C^+$.
Notice $f_c$ is continuous (even analytic) and strictly concave in $M$ and it attains its maximal value at $-c$. Also notice that $\lambda\in C^-$ so that $-c \in C_\lambda$. 

We will check $F(x) <\beta$ for $x$ in different parts
of $\mathbb{T}\setminus C_\lambda$. 
Since $[-c, -c +q^{-1}]$ is of length $q^{-1}$, for any
$x \in J^+$ there exists a unique $x_0\in [-c, -c +q^{-1})$ such that $q(x-x_0)\in \mathbb{Z}$. Similarly, for any
$x \in J^-$ there exists a unique $x_0\in (-c-q^{-1}, -c]$ such that $q(x-x_0)\in \mathbb{Z}$. We will estimate $F(x)<\beta$ for $x \in J^+$ by $F(x_0)$ for some $x_0$ in $[-c, -c +q^{-1})$ (the right half of $M$), and for $x \in J^-$ by $F(x_0)$ for some $x_0$ in $(-c-q^{-1}, -c)$ (the left half of $M$). The interval
$(-c -q^{-1}, -c]$ will be cut into two by $\lambda$ and
the interval
$[-c,  -c+q^{-1})$ will be cut into two by $\lambda+q^{-1}$.

%Notice that $\mathbb{T}$ is the disjoint union of $J^-, C_\lambda$
%and $J^+$. 
\medskip

We shall consider the following three cases:

{\em Case I.} \ \ $x \in C^-\cup C^+$. %or $-c-q^{-1}<x<\lambda$,
%|x+c|<\frac{1}{q}$ and $x\not\in [\lambda, \lambda+q^{-1}]$,

%For each $x$ with $0<x+c\le \frac{1}{2}$, let $x_0$ be such that $q(x-x_0)\in \Z$ and $0\le x_0+c<\frac{1}{q}$.
%For each $x$ with $-\frac{1}{2}<x+c<0$, let $x_0$ be such that $q(x-x_0)\in \Z$ and $0\ge x_0>-\frac{1}{q}$.

{\em Case II.}\  $x\in J^+$ and $x_0\in [-c, \lambda+q^{-1}]$;
or $x\in J^-$ and $x_0\in [\lambda, -c]$.
%$$-c\le x_1<\lambda+q^{-1}$; or $-\frac{1}{2}< x+c<-\frac{1}{q}$ and $\lambda <x_1\le -c$.

{\em Case III.} $x\in J^+$ and  $x_0\in (\lambda+q^{-1}, -c+q^{-1})$;
or $x\in J^-$ and  $x_0\in (-c-q^{-1},\lambda)$.
%$\lambda+q^{-1}\le x_1<-c+q^{-1}$; or
%$-\frac{1}{2}< x+c<-\frac{1}{q}$ and $\lambda-q^{-1} <x_1< \lambda$.
\medskip

Note that if $q=2$, then $M=\mathbb{T}\setminus \{-c -1/2\}$ and we only need  to consider Case I, 
because $F(-c-1/2)=-\infty$. Similarly, if $q=3$ then we only need to consider Case I and Case II.

%Note that Case II only happens when $q\ge 3$ and Case III only when $q\ge 4$.
%In the following, we shall only give details for $x\in (\lambda+1/q, -c+1/2)$, as the other case is similar.

Before going further, let us state two useful elementary facts.

\begin{lemma}\label{ineq:E1} Given $\alpha\in (0,1)$, the function $h(x)=\frac{\sin x}{\sin \alpha x}$ is strictly decreasing in $(0, \pi)$.
\end{lemma}
\begin{proof} We can continuously extend $h$ on $0$ by $h(0)=1/\alpha$ and we have $h'(0)=0$. By direct computation, $$h'(x)=\frac{\sin (\alpha x)\cos x -\alpha \sin x \cos \alpha x}{\sin^2 (\alpha x)},$$
	$$(\sin^2 (\alpha x) h'(x))'=(\alpha^2-1) \sin x \sin (\alpha x)<0.$$ Therefore $h'(x)<0$ on $(0, \pi)$, which implies that $h$ is strictly decreasing.
\end{proof}

\begin{lemma} \label{lem:AksAs}
	For any $q\ge 2$, any integer $1\le k\le\frac{q-1}{2}$ and any $s\in (0, q^{-1})$, we have
	$$\sin \pi (s+k\cdot q^{-1})\ge \sin \pi (s+q^{-1}).$$
\end{lemma}
\begin{proof}
	This is of course true for $k=1$. So assume $k\ge 2$ which implies that $q\ge 5$.
	Notice that $\sin \pi x$ is increasing on $[0,1/2]$
	and symmetric about $x=1/2$.
	Then the announced  inequality holds because
	$$0<s+q^{-1}<\min \left(\frac{1}{2}, s+k\cdot q^{-1}\right)$$
	and $$s+q^{-1}+s +k \cdot q^{-1}\le \frac{k+3}{q}\le \frac{\frac{q-1}{2}+3}{q}\le 1.$$
\end{proof}

%Since $f_c'$ is decreasing in $C_\lambda$, the equation (\ref{eqn:varphi'}) immediately implies that
%\begin{equation}\label{eqn:varphixy0}
%\psi(y)-\psi(x)\le f_c'(\lambda)\frac{y-x}{q-1}=f'(\theta-q^{-1})\frac{y-x}{q-1},
%\end{equation}
%for any $x<y$. If $(x,y)$ is an interval disjoint from $C_\lambda$, we have the following improvement:
\subsection{Variation of $\psi$}
The following lemmas give us the estimates for the variations of $\psi$ and $\psi\circ T$.  Put
$$\theta:=\lambda+1/q+c\in (0, 1/q).$$

\begin{lemma}\label{lem:varphixy}\ \, \\
\indent {\rm (i)}\ For any $\lambda+q^{-1}\le x<y\le \lambda+1$, we have
$$\psi(y) -\psi(x)\le f\left(q^{-1}-\theta-\frac{y-x}{q-1}\right)-f\left(q^{-1}-\theta\right)\le f(0)-f(q^{-1}-\theta).$$
\indent {\rm (i)'}\ For any $\lambda+q^{-1}-1 \le y<x\le \lambda$, we have 
$$\psi(y)-\psi(x)\le f\left(\theta-\frac{x-y}{q-1}\right)-f\left(\theta\right)\le f(0)-f(\theta).$$ 
\indent {\rm (ii)} For any $x<y$ with $y-x<1$, we have
\begin{eqnarray*}\psi(y)-\psi(x)
	 & \le & f\left(q^{-1}-\theta-\frac{y-x}{q}\right)-f\left(q^{-1}-\theta\right)-f'(q^{-1}-\theta) \frac{y-x}{q(q-1)}\\
	 & \le & -f'(q^{-1}-\theta)\frac{y-x}{q-1}. 
\end{eqnarray*}
\indent {\rm (ii)'} For $y<x$ with $x-y<1$, we have 
\begin{eqnarray*}
\psi(y)-\psi(x) 
& \le & f\left(\theta-\frac{x-y}{q}\right)-f\left(\theta\right)-f'(\theta)\frac{x-y}{q(q-1)}\\
& \le & -f'(\theta) \frac{x-y}{q-1}.
\end{eqnarray*}
%\marginpar{changed}
\end{lemma}
\begin{proof} We shall only prove (i) and (ii) and leave the analogous (i)' and (ii)' for the reader. Let $J=(x, y)$ and for each $n\ge 0$, $J_n:=\tau_{\lambda(c)}^n(J)$. By the formula (\ref{eqn:varphi'}), we have
$$\Delta:=\psi(y) -\psi(x)=\sum_{n=1}^\infty \int_{J_n} f_c'(x) dx.$$

(i) The second inequality is obvious because $f$ attains its maximal value at $0$. Let us prove the first inequality. Since $J\cap C_\lambda'=\emptyset$ and $\tau_{\lambda(c)}(\T)\subset C_\lambda'$, $J_n$'s ($n\ge 1$) are disjoint sets contained in $C_{\lambda}'$.
Together with the fact that $f'_c$ is decreasing in $C_{\lambda}$, we immediately obtain the following estimate:
$$\Delta\le \int_{\lambda}^{\lambda+(y-x)/(q-1)} f_c'(t)dt=f\left(\theta-q^{-1}+\frac{y-x}{q-1}\right)-f\left(\theta-q^{-1}\right).$$
Since $f$ is an even function, the desired inequality follows.

(ii) Since $J_1$ is contained in $C_\lambda$
and $f_c'$ is decreaing in $C_\lambda$,   we have $$\int_{J_1}f_c'(t)dt\le \int_{\lambda}^{\lambda+(y-x)/q} f_c'(t)dt= f\left(\theta-q^{-1}+\frac{y-x}{q}\right)-f\left(\theta-q^{-1}\right),$$
and for each $n\ge 2$, we simply estimate $$\int_{J_n} f_c'(t) dt \le f_c'(\lambda)|J_n|=f'(\theta-q^{-1}) (y-x)/q^n.$$
The first inequality follows.
The second inequality holds because for any $-\theta<u<q^{-1}-\theta$, $f'(u)\ge f'(q^{-1}-\theta)$.
\end{proof}

\begin{lemma}\label{lem:varphiTxTgamma}
	\ \,\\
\indent  {\rm (i)}\ For any $x\in (\lambda+q^{-1},\lambda+2q^{-1})$ and $t:=x-\lambda-q^{-1}$, we have
$$\psi(T(x))-\psi(T(\lambda+q^{-1}))\ge f(q^{-1}-\theta-t)-f(q^{-1}-\theta)+f'(\theta)\frac{t}{q-1}.$$
\indent  {\rm (i)'} For $x\in (\lambda-q^{-1},\lambda)$ and $t:=\lambda-x$, we have
$$\psi(T(x))-\psi(T(\lambda)) \ge f(\theta-t)-f(\theta)+f'(q^{-1}-\theta)\frac{t}{q-1}.$$
\end{lemma}

\begin{proof} We only deal with (i). Put $\lambda^*=\lambda+q^{-1}$
	which is the right end point of $C_\lambda$, $J=(\lambda^*,x)$ and $\Psi=\psi\circ T$. We have $t=|J|$. For any $y\in J$, $\tau(Ty)=y-q^{-1}\in C_\lambda'$. Hence
\begin{multline*}
\Psi'(y)=q\psi'(Ty)=q\sum_{n\ge 1} f_c'(\tau^n(Ty))q^{-n}\\
=f_c'(y-1/q)+\psi'(y-q^{-1})\ge f_c'(y-q^{-1})+\frac{f_c'(\lambda)}{q-1},
\end{multline*}
where,  for the last inequality, we used the formula (\ref{eqn:varphi'}), the facts $\sum_{n=1}^\infty |J_n|=\frac{1}{q-1}$ and  $f_c'$ is decreasing in $C_\lambda$.
Therefore, integrate to get
$$\psi(T(x))-\psi(T(x_0))\ge \int_J f_c'(y-q^{-1}) dy + \frac{t f_c'(\lambda)}{q(q-1)},$$
which is equivalent to the desired inequality.
\end{proof}
%\end{document}
\subsection{Proof of $F(x)<\beta$ in Case I}
We deal with Case I in this subsection. The argument is motivated by Jenkinson \cite{Jenkinson2007}.

\begin{prop} \label{prop:caseI}
For $x\in C^-\cup C^+$, we have $F(x)< \beta$.
\end{prop}
\begin{proof} We only deal with the case $x\in C^+= (\lambda+q^{-1}, -c+q^{-1})$ as the other case is similar.
Put $\lambda^*=\lambda+q^{-1}$ and $t=x-\lambda^*(>0)$.
Write
$$
 F(x)-\beta =F(x)-F(\lambda^*)\\
  =f_c(x)-f_c(\lambda^*)+\psi(x)-\psi(\lambda^*)-(\psi(T(x))-\psi(T(\lambda^*)).
$$
%Since $f_c$ is decreasing in $[-c, -c+q^{-1}]$, we have
Notice that $x+c = x-\lambda^* + \theta$, we have
$$
 f_c(x) - f_c(\lambda) = f(\theta +t) -f(\theta).
$$
%, we have
%$f_c(x) - f_c(\lambda) = f(\theta +t) -f(\theta)$.
By Lemmas~\ref{lem:varphixy} (i) and~\ref{lem:varphiTxTgamma} (ii),
we have
$$
\psi(x)-\psi(\lambda^*) \le f\left(q^{-1}-\theta-\frac{t}{q-1}\right)-f(q^{-1}-\theta)
$$
and
$$
  \psi(T(x))-\psi(T(\lambda^*))\ge
  f(q^{-1}-\theta-t)-f(q^{-1}-\theta)+f'(\theta) \frac{t}{q-1}
$$
Therefore $F(x) -\beta$ is bounded by
\begin{align*}
%F(x)-\beta &\le
f(\theta +t) -f(\theta) + %(f(\theta+t)-f(\theta)) +
f(q^{-1}-\theta-\frac{t}{q-1})-f(q^{-1}-\theta-t)-f'(\theta)\frac{t}{q-1}.
\end{align*}
The sum of the third and the forth terms are strictly negative,
because  $f$ is strictly decreasing in $(0,q^{-1})$ and $t>0$ and $t+\theta<q^{-1}$, so that $q^{-1}>q^{-1}-\theta-\frac{t}{q-1}> q^{-1}-\theta-t>0$. Thus we get
$$F(x)-\beta< f(\theta+t)-f(\theta)-f'(\theta) \frac{t}{q-1}=:H(\theta,t).$$
Since $$\frac{\partial H}{\partial t}=f'(\theta+t)-\frac{1}{q-1}f'(\theta)<  \frac{q-2}{q-1} f'(\theta)\le 0,$$
we conclude that $F(x)-F(\beta)< H(\theta,0)=0$.
%Note that
%\begin{align*}
%& \frac{\partial H}{\partial t}\\
%=& f'(\theta+t)-\frac{1}{q-1}f'(q^{-1}-\theta-\frac{t}{q-1})+f'(q^{-1}-\theta-t)-\frac{1}{q-1} f'(\theta)\\
%= & \left(f'(\theta+t)-f'(\theta-q^{-1}+t)\right)+\frac{1}{q-1} \left(f'(-\theta)-f'(-\theta+q^{-1}-\frac{t}{q-1}\right)\\
%&
%\end{align*}
%By the mean value theorem, there exist $\theta_1>\theta$ and $\theta_2<\theta+t-1/q<\theta_1$ such that
%$$F(x)-\beta=
%f'(\theta_1)t-f'(\theta_2)(t-t/(q-1))-f'(\theta) t/(q-1).$$
%Since $f'$ is decreasing in $(-1/q, 1/q)$, we conclude that $F(x)-\beta<0$.
\end{proof}
Note that the proposition above completes the proof of the theorem in the case $q=2$.

\subsection{Proof of $F(x)<\beta$ in Case II}
%For $q\ge 3$, we have to consider the case where $(-c,x)$ contains singularities of $f_c$. Taking advantage of the form $f$, we see that the %interval $C_\lambda$ is almost centered at $-c$.
The following estimates of $\theta$ are needed in the proofs in Case II and Case III. Recall that $\theta = \lambda^* + c = \lambda + q^{-1} +c \in (0, q^{-1})$.

%To deal with this case and the next case, we shall need the %following

\begin{lemma}\label{lem:almostcentered} Assume $q\ge 3$. Then
$$\frac{3}{8q}<\theta<\frac{5}{8q}.$$
\end{lemma}
\begin{proof}
Without loss of generality, we assume that $0<\theta\le 1/2q$.  Since $f_c'$ is decreasing in $C_\lambda$, we have
$$\int_{\tau_\lambda^{n-1}(C_\lambda)}f_c'(x) dx \ge \int_{\lambda+q^{-1}-q^{-n}}^{\lambda+q^{-1}} f_c'(x) dx=f(\theta)-f(\theta-q^{-n}).$$
By (\ref{eqn:preStur}), we obtain
$$D(\theta):=\sum_{n=1}^\infty \left(f(\theta)-f(\theta-q^{-n})\right)\le 0.$$
Since $f'$ is a smooth and strictly decreasing function in $(-q^{-1}, q^{-1})$,
$$D'(x)=\sum_{n\ge 1} (f'(x)-f'(x-q^{-n}))<0,$$
for all $x\in [0,q^{-1})$.
Therefore, $D$ is strictly decreasing in $[0, q^{-1})$ and
it suffices to check $D(3/8q)>0$, i.e.
$$\Delta:=\sum_{n=1}^\infty \left(f\left(\frac{3}{8q}\right)-f\left(\frac{3}{8q}-q^{-n}\right)\right)>0.$$
Indeed, if $q\ge 4$, by the mean value theorem we have
\begin{align*}
\Delta\ge & f\left(\frac{3}{8q}\right)-f\left(-\frac{5}{8q}\right) +\sum_{n\ge 2} f'\left(\frac{3}{8q}\right) q^n\\
=& \log \left(\frac{\sin \frac{5\pi}{8q}}{\sin \frac{3\pi}{8q}}\right)+ \pi \left(q\tan \frac{\pi}{8}-\cot \frac{3\pi}{8q}\right) \frac{1}{q(q-1)}\\
>& \log \left(\frac{\sin \frac{5\pi}{8q}}{\sin \frac{3\pi}{8q}}\right)+ \pi\left(\tan\frac{\pi}{8}-\frac{8}{3\pi}\right) \frac{1}{q-1}\\
> & \log \frac{\sin \frac{5\pi}{32}}{\sin \frac{3\pi}{32}}+\frac{\pi}{3}\left(\tan \frac{\pi}{8}-\frac{8}{3\pi}\right)>0,
\end{align*}
where we have used Lemma \ref{ineq:E1}, the inequality
$\cot x \le x^{-1}$ over $(0, \pi/2)$ and the fact $q\ge 4$, and the last inequality can be numerically checked;
and if $q=3$, then $f(1/8)=0.8813...$, $f(-5/24)=0.4171...$, $f(1/72)=1.0960...$, $f'(1/8)=-3.6806...$, and hence
\begin{align*}
\Delta\ge  &f\left(\frac{1}{8}\right)-f\left(-\frac{5}{24}\right)+ f\left(\frac{1}{8}\right)-f\left(\frac{1}{72}\right)+\sum_{n\ge 3} f'\left(\frac{1}{8}\right) 3^{-n}\\
\ge & 0.881-0.418+0.881-1.097 -3.681\cdot \frac{1}{18}>0.
\end{align*}
\end{proof}

\iffalse
\begin{lemma} \label{lem:AksAs}
For any $q\ge 2$, any integer $1\le k\le\frac{q-1}{2}$ and any $s\in (0, q^{-1})$, we have
$$\sin \pi (s+k\cdot q^{-1})\ge \sin \pi (s+q^{-1}).$$
\end{lemma}
\begin{proof}
This is of course true for $k=1$. So assume $k\ge 2$ which implies that $q\ge 5$. Then the inequality holds because
$$0<s+q^{-1}<\min \left(\frac{1}{2}, s+k\cdot q^{-1}\right)$$
and $$s+q^{-1}+s +k \cdot q^{-1}\le \frac{k+3}{q}\le \frac{\frac{q-1}{2}+3}{q}\le 1.$$
\end{proof}
\fi

The following technical lemma is based on numerical calculation, which is needed to complete the proof in Case II.
\begin{lemma} \label{lem:CaseIInum}
Given $q\ge 3$, the following holds for all $t\in \left(\frac{3}{8q},\frac{5}{8q}\right)$ and all $0< s\le q^{-1}-t$:
$$H(t,s):=A(s)+B(t,s)<0,$$
where
$$A(s):=f(s+q^{-1})-f(s)=\log \frac{\sin \pi s}{\sin \pi(q^{-1}+s)},$$
and
$$B(t,s):=f(0)-f(t)-f'(t) \frac{q^{-1}-t-s}{q-1}.$$
\end{lemma}
\begin{proof} Let $U=\{(t,s): \frac{3}{8q}\le t\le \frac{5}{8q}, 0<s<q^{-1}-t\}$, a trapezoid in the plane. Then for any $(t,s)\in U$,
%$$\frac{\partial H}{\partial t}(t,s)=\frac{q-2}{q-1}\left(f'(-t)-f'\left(-t+\frac{t+s}{q-1}\right)\right)-f''(-t)\frac{q^{-1}-t-s}{q-1}>0.$$
$$\frac{\partial H}{\partial t}(t,s)=-\frac{q-2}{q-1}f'(t)-f''(t)\frac{q^{-1}-t-s}{q-1}>0.$$
%Note that $\partial H/\partial t (t,s)>0$ whenever $0<t<q^{-1}$ and $0<s\le q^{-1}-t$.
%Since
%\begin{align*}
%A'(s)& =\pi (\cot \pi s -\cot \pi(q^{-1}+s))\\
%& =\pi \frac{\sin \frac{\pi}{q}}{\sin (\pi s)\sin (\pi(s+q^{-1}))}\\
%&\le \frac{\sin \frac{\pi}{q}}{\sin \frac{5q}{8}}
%\end{align*}
%and $$f'(-t)\le f'(-5/8q)=\pi q(\tan \frac{\pi}{8}+ q^{-1}\cot \frac{5\pi}{8q})\le \pi q(\tan\frac{\pi}{8}+ \frac{8}{5\pi}),$$
%we have
%\begin{align*}
%\frac{\partial H}{\partial s} =A'(s)-\frac{1}{q-1}f'(-t)
%\ge A'(s)-\frac{1}{q-1}f'(s-q^{-1})=:G(s).
%$$G(s)=\pi (\cot \pi s-\cot \pi (q^{-1}+s))-\frac{\pi}{q-1} (q\cot (\pi qs)-\cot (\pi (s-q^{-1}))).$$
%
%$$G'(s)=A''(s)-\frac{1}{q-1}f''(s-q^{-1})=-\pi^2\left(\frac{1}{\sin^2 (\pi s)}-\frac{1}{\sin^2 %(\pi(s+q^{-1}))}\right)-\frac{\pi^2}{q-1}\left(\frac{q^2}{\sin^2 (\pi qs)}-\frac{1}{\sin^2 (\pi s)}\right)$$
%
Thus, as function of $t$, $H(t, s)$ is increasing, and it suffices to check that $H$ is negative on the 
right-hand-side part of the boundary of $U$, i.e.
\begin{enumerate}
\item [(i)] $H(5/(8q), s)<0$ for all $0<s\le 3/(8q)$.
\item [(ii)] $H(q^{-1}-s,s)<0$ for $s\in [3/(8q), 5/(8q)]$.
\end{enumerate}

%Let
%$$B(t, s)=f\left(-t+\frac{t+s}{q-1}\right)-f\left(-t\right)+f'\left(-t\right)\frac{q^{-1}-t-s}{q-1}.$$
%$$B(t, s)=f\left(0\right)-f\left(-t\right)+f'\left(-t\right)\frac{q^{-1}-t-s}{q-1}.$$
%and
%$$\tilde{B}(t)=f'(-t)\frac{1}{q(q-1)}.$$
Note that $$
A'(s)=\frac{\pi}{\sin \pi s}- \frac{\pi}{\sin\pi (s+q^{-1})}>0,
\quad  \frac{\partial B}{\partial s}(t,s) = \frac{f'(t)}{q-1}<0.
$$ %and $B(t, s)\le \tilde{B}(t)$.

Let us prove (i). First assume $s\le \frac{1}{4q}$. In this case, we use
$$H(5/8q, s)\le A(1/(4q))+ B(5/8q,0).$$
%f'(-\frac{5}{8q}) \frac{1}{q(q-1)}.$$
Since
$$A(1/4q)=-\log \frac{\sin 5\pi/(4q)}{\sin \pi /(4q)} \le -\log \frac{\sin 5\pi/12}{\sin \pi/12}\le -1.3169...,$$
and
$$f(0)-f(\frac{5}{8q})=\log \frac{q\sin \frac{5\pi}{8q}}{\sin \frac{5\pi}{8}}\le \log \frac{5\pi/8}{\sin (5\pi/8)}=0.7538...,$$
$$-f'\left(\frac{5}{8q}\right)=\pi q(\tan \frac{\pi}{8}+ q^{-1}\cot \frac{5\pi}{8q})\le \pi q(\tan\frac{\pi}{8}+ \frac{8}{5\pi})<2.91 q,$$
% \log \frac{\sin 11\pi/32}{\sin 3\pi/32}>1,$$
we obtain
$$H(5/8q, s)\le -1.3169+ 0.7539 + 2.91 q \cdot \frac{3}{8q(q-1)}<-0.01<0.$$
Now assume $1/(4q)<s<3/(8q)$. Then
$$H(5/(8q), s)\le A(3/8q)+ B(5/8q, 1/4q).$$
Since $$A(3/8q)=-\log \frac{\sin 11\pi/(8q)}{\sin 3\pi/(8q)}\le   -\log\frac{\sin \frac{11\pi}{24}}{\sin \frac{\pi}{8}}=-0.9519....,$$
we obtain
$$H(5/8q, s)\le -0.9519+0.7539+2.91 q\cdot \frac{1}{8q(q-1)}<-0.01<0.$$

%$f(-5/48)=0.9503...$, $f(-5/24)=0.4171...$, $f'(-5/24)=7.9980...$, and
%$$A(1/8)\le   -\log\frac{\sin \frac{11\pi}{24}}{\sin \frac{\pi}{8}}=-0.9519....$$ Thus
%$$H(5/24, s)\le A(1/8)+ f(-5/48)-f(-5/24)+\frac{f'(-5/24)}{16}<0.$$  (not enough)

%$$B:=\frac{f'(-5/8q)}{q(q-1)}=\frac{\pi}{q(q-1)}\left(q\cot (-\frac{5\pi}{8})-\cot (\frac{-5\pi}{8q}) \right)=\frac{\pi}{q-1}\tan \frac{\pi}{8} %+\frac{\pi}{q(q-1)} \cot \frac{5\pi}{8q}$$
%\begin{align*}
%H(3/8q, s)& \le -\log \frac{|\sin \pi(s+q^{-1})|}{|\sin \pi s|} + \frac{f'(-5/8q)}{q(q-1)}\\
%&\le -\log \frac{\sin 11\pi/8q}{\sin 3\pi/8q}+\frac{\pi}{q(q-1)}\left(q\cot (-5\pi/8)-\cot (-5\pi/8q) \right)\\
%&\le -\log \frac{\sin 11\pi/24}{\sin \pi/8} +\frac{\pi}{6} \left(3\tan (\pi/8)+\cot (5\pi/24)\right)<0.
%\end{align*}

Finally, let us prove (ii). If $3/(8q)<s\le 1/(2q)$, then
\begin{align*}
H(q^{-1}-s,s)& \le  A(1/2q)+B(5/8q, 3/8q)\\
 &\le A(1/2q)+f(0)-f(-5/8q)\\
&= \log \frac{q\sin \frac{\pi}{2q}\sin \frac{5\pi}{8q}}{\sin\frac{3\pi}{2q}\sin\frac{5\pi}{8}}<0,
\end{align*}
where the last inequality holds because for $q=3$, we check directly; for $q\ge 4$, we have
$$\frac{q\sin \frac{\pi}{2q}\sin \frac{5\pi}{8q}}{\sin\frac{3\pi}{2q}\sin\frac{5\pi}{8}}\le \frac{q\cdot \frac{\pi}{2q}\sin\frac{5\pi}{8q}}{\sin \frac{5\pi}{8}\sin \frac{10\pi}{8q}}\le \frac{\pi}{4\sin \frac{5\pi}{8}\cos \frac{5\pi}{8q}}\le \frac{\pi}{4\sin \frac{5\pi}{8}\cos \frac{5\pi}{32}}<1.$$
If $1/(2q)<s<5/(8q)$, then
\begin{align*}
H(q^{-1}-s,s)& \le A(5/8q)+ B(1/2q, 1/2q)\\
& \le A(5/8q)+ f(0)-f(-1/2q)\\
& =\log \frac{q\sin\frac{5\pi}{8q}\sin\frac{\pi}{2q} }{\sin \frac{13\pi}{8q}}<0,
\end{align*}
where the last inequality can be checked directly.
\end{proof}

\begin{prop}\label{prop:neigh}
Suppose that we are in Case II. %If $q^{-1}<x+c<q^{-1}+\theta$ or $\theta-2 q^{-1} <x+c<-q^{-1}$,
Then $F(x)<\beta$.
\end{prop}
\begin{proof} Once again we only deal with the case $x>-c$, as the other case is similar.
Let $x_0$ be the unique point in $[-c,\lambda+q^{-1})$ with $q(x-x_0)=:k\in \Z_+$, let $\lambda^*=\lambda+q^{-1}$. We may assume that $x_0\not=-c$ for otherwise $F(x)=-\infty$. Let $$
s=x_0+c\in (0,\theta), \quad t=q^{-1}-\theta\in (0, q^{-1}),
$$ so $\lambda^*-x_0=q^{-1}-s-t$. It suffices to prove that $F(x)<F(x_0)=\beta$. By Lemma~\ref{lem:varphixy} (i) and (ii),
$$ \psi(x)-\psi(\lambda^*)\le f(0)-f(t),\quad
    \psi(\lambda^*)-\psi(x_0)\le -f'(t)\frac{q^{-1}-s-t}{q-1}.
$$
Thus
\begin{align*}
\Delta:& =F(x)-F(x_0)\\
&=f(s+k q^{-1})-f(s)+\psi(x)-\psi(\lambda^*)+\psi(\lambda^*)-\psi(x_0)\\
&\le f(s+k q^{-1})-f(s)+B(t,s)\\
& \le A(s) + B(t,s)=H(t,s),
\end{align*}
where we used Lemma~\ref{lem:AksAs} to obtain the last inequality.
We can apply Lemma~\ref{lem:AksAs}, because
$x_0 +kq^{-1}\le \lambda+(q-1)/(2q)$, which implies $k\le (q-1)/2$.
As Case II only happens when $q\ge 3$, by Lemma~\ref{lem:almostcentered}, we have $\theta \in (3/(8q), 5/(8q))$
then  $t\in (3/(8q), 5/(8q))$.
The proof is completed by Lemma~\ref{lem:CaseIInum}.
\end{proof}

\iffalse
\begin{lemma} Fix an integer $q\ge 3$. If $3/(8q)<t<5/(8q)$ and $0<s<q^{-1}-t$, then $H(t, s)<0$.
\end{lemma}
\fi

\subsection{Proof of $F(x)<\beta$ in Case III}
The following lemma is based on numerical calculation which is needed to complete the proof in Case III.

\begin{lemma}\label{lem:CaseIIInum} Let $q\ge 4$ be given. For $t\in (3/8q, 5/8q)$, $t\le s<1/q$, we have
\begin{equation}
\label{eqn:dfnG}
G(t, s)=U(s)+V(t, s)<0,
\end{equation}
where
$$U(s)=\log \frac{\sin \pi(q^{-1}-s)}{\sin \pi (q^{-1}+s)}$$
and
$$V(t, s)=f(0)-f(t)+f\left(\frac{1}{q}-t-\frac{s-t}{q-1}\right)\\
-f\left(\frac{1}{q}-t\right)-f'(t)\frac{s-t}{q-1}.$$
\end{lemma}
\begin{proof} It suffices to check that $G(t,t)<0$ and $\partial G/\partial s<0$.
\begin{align*}
\frac{\partial G}{\partial s}(t,s) &= U'(s)+\frac{1}{q-1}\left(f'\left(t-\frac{1}{q}+\frac{s-t}{q-1}\right)-f'(t)\right)\\
& \le U'(s)+\frac{1}{q-1} \left(f'(t-q^{-1})-f'(t)\right).
\end{align*}
But
$$U'(s)= -\pi\left(\cot \pi(q^{-1}-s)+\cot (q^{-1}+s)\right)\le -\pi \cot \pi(q^{-1}-t),$$
%\left(\cot \pi(q^{-1}-t)+\cot \pi(q^{-1}+t)\right),$$
$$f'(t-q^{-1})-f'(t)=\pi\left(\cot \pi(q^{-1}-t)+\cot \pi t\right).$$
When $q\ge 4$, $t\in (3/8q, 5/8q)$ we have $\cot \pi t< 2 \cot \pi (q^{-1}-t)$, so $\partial G/\partial s<0$.
%follows
%$$\frac{\partial G}{\partial s}\le -\frac{(q-2)\pi}{q-1}\left(\cot \pi(q^{-1}-t)+\cot \pi t\right)<0.$$
On the other hand,
\begin{align*}
G(t,t)&= U(t)+f(0)-f(t)\\
&= \log \frac{q\sin \pi(q^{-1}-t)\sin \pi t}{\sin \pi (\cdot q^{-1}+t)\sin \pi q t}\\
& \le \log \frac{q\sin \pi(q^{-1}-t)\sin \pi t} {\sin \frac{11\pi}{8q} \sin \frac{3\pi}{8}}\\
&\le \log \frac{q\sin^2 \frac{\pi}{2q}}{\sin \frac{11\pi}{8q}\sin \frac{3\pi}{8}}
<0,
\end{align*}
where the last inequality holds because: $q\sin \frac{\pi}{2q}<\frac{\pi}{2}$, and when $q\ge 4$, by Lemma \ref{ineq:E1}
$$\frac{\sin \frac{\pi}{2q}}{\sin \frac{11\pi}{8q}}\le \frac{\sin \frac{\pi}{8}}{\sin \frac{11\pi}{32}}=0.4339...,$$ so
$$\frac{q\sin ^2 \frac{\pi}{2q}}{\sin \frac{11\pi}{8q}\sin \frac{3\pi}{8}}=q\sin \frac{\pi}{2q} \cdot \frac{1}{\sin \frac{3\pi}{8}}\cdot
\frac{\sin \frac{\pi}{2q}}{\sin \frac{11\pi}{8q}}\le \frac{\pi}{2}\cdot \frac{1}{\sin \frac{3\pi}{8}}\cdot 0.434<1.$$
\end{proof}

\begin{prop} %Suppose that we are in Case III. %For any $\lambda+2/q\le x\le -c+ 1/2$ or $-c-1/2\le x\le \lambda-q^{-1}$,
We have  $F(x)<\beta$ in Case III.
\end{prop}
\begin{proof} Once again, we shall only give details for the case when $x>-c$. let $\lambda^*=\lambda+1/q$. Take $x_0\in (\lambda^*, -c+1/q)$ such that $k=q(x-x_0)$ is a positive integer. We can of course assume $x_0\not=-c$ for otherwise $F(x)=-\infty$. Put $s=x_0+c\in (\theta,1/q)$ so that $x_0-\lambda^*=s-\theta$ and $x+c=s+k/q$.
Since this case only happens for $q\ge 4$, by Lemma~\ref{lem:almostcentered}, we have $\theta\in (3/8q, 5/8q)$. So it suffices to prove
\begin{equation}\label{eqn:G}
F(x)-\beta<G(\theta, s)
\end{equation}
and then apply Lemma~\ref{lem:CaseIIInum}.

Let us prove (\ref{eqn:G}). We have
$f(x+c) = f(s+k\cdot q^{-1})$ and
$$f(s+k\cdot q^{-1})-f(s+q^{-1})=-\log \frac{\sin \pi(s+k\cdot q^{-1})}{\sin \pi(s+q^{-1})}<0,$$
which is a consequence of Lemma~\ref{lem:AksAs}.
Hence
\begin{equation}\label{eqn:G1}
f(x+c)-f(q^{-1}-s)\le f(q^{-1}+s)-f(q^{-1}-s)=U(s).
\end{equation}
By Lemma~\ref{lem:varphixy} (i),
\begin{equation}\label{eqn:G2}
\psi(x)-\psi(x_0)\le f\left(0\right)-f(q^{-1}-\theta),
\end{equation}
\begin{equation}\label{eqn:G3}
\psi(x_0)-\psi(\lambda^*)\le f\left(q^{-1}-\theta- \frac{s-\theta}{q-1}\right)-f(q^{-1}-\theta).
\end{equation}
By Lemma~\ref{lem:varphiTxTgamma} (i),
\begin{equation}\label{eqn:G4}
\psi(T(x_0))-\psi(T(\lambda^*))\ge f(q^{-1}-s)- f(q^{-1}-\theta)+f'(\theta) \frac{s-\theta}{q-1}.
\end{equation}
%Let us first consider the case $x_1\ge x_0$, i.e. $s\ge \theta$.
Note that $f(\lambda^*+c)=f(\theta)$. Therefore,
\begin{multline*}
F(x)-\beta=F(x)-F(x_0)+F(x_0)-F(\lambda^*)\\
=\left(f_c(x)-f_c(\lambda^*)\right)+\left(\psi(x)-\psi(x_0)\right)+\left(\psi(x_0)-\psi(\lambda^*)\right)-\left(\psi(T(x_0))-\psi(T(\lambda^*))\right)\\
\le \left(f(x+c)-f(\theta)\right)+ \left(f(0)-f(q^{-1}-\theta)\right)
+\left(f\left(q^{-1}-\theta- \frac{s-\theta}{q-1}\right)-f(q^{-1}-\theta)\right)\\
-\left(f(q^{-1}-s)- f(q^{-1}-\theta)+f'(\theta) \frac{s-\theta}{q-1}\right)\\
= \left(f(x+c)-f(q^{-1}-s)\right)+V(\theta,s)
\le U(s)+V(\theta,s)=G(\theta,s),
\end{multline*}
where we used (\ref{eqn:G2}), (\ref{eqn:G3}) and (\ref{eqn:G4}) for the first inequality and (\ref{eqn:G1}) for the second inequality.
\end{proof}

\section{Appendix A: $q$-Sturmian measures}
In this section we give a proof of the existence and uniqueness of $q$-Sturmian measures and review some relevant facts.   
Throughout fix an integer $q\ge 2$ and let $T:\T\to \T$ denote the circle map $x\mapsto qx\mod 1$.

\begin{prop}\label{prop:qStur}
 %Fix $q\ge 2$ and consider the dynamics on $\mathbb{T}$ defined by $T(x)=qx\mod 1$. 
 For each $\lambda\in \R$, there is a unique $T$-invariant Borel probability measure $\frak{S}_\lambda$ supported in $C_\lambda=[\lambda,\lambda+q^{-1}] \mod 1 \subset \mathbb{T}$. We have $\frak{S}_\lambda=\frak{S}_{\lambda+1}$ for each $\lambda\in \R$.
 Moreover, putting 
 $$\Gamma:=\{\lambda\in [0,1): \frak{S}_\lambda \text{ is NOT supported on a periodic orbit}\},$$
 then $\overline{\Gamma}$ has Hausdorff dimension zero. 
 %
 %there is a continuous, monotone map $\rho: \R\to \R$ such that\\
%\begin{itemize}
%\indent {\rm (1)} $\rho(\lambda+1)=\rho(\lambda)+q-1
%
%\indent {\rm (1)} $\lambda\not\in \Gamma$ if and only if $\frak{S}_\lambda$ is supported on a periodic orbit.\\
%\indent {\rm (2)} for each $r\in \Q/\Z$, $\rho^{-1}(r)$ is a non-degenerate interval.\\
%\indent {\rm (2)}  $\overline{\Gamma}$ has Hausdorff dimension zero.
%\end{itemize}
\end{prop}
For each $\lambda\in \R$, let $R_\lambda: \T\to \T$ denote the continuous map which satisfies that $R_\lambda|_{C_\lambda}=T|_{C_\lambda}$ and $R_\lambda$ is constant in $\T\setminus C_\lambda$. So $R_{\lambda+1}=R_\lambda$ for each $\lambda$. The map $R_\lambda$ is a monotone continuous circle map of degree one and it has a well-defined rotation number $\rho(\lambda)\in \T$. Since $R_\lambda(x)$ is continuous in $(x,\lambda)$, $\rho(\lambda)$ is continuous in $\lambda$. For each $x\in \T$, $\lambda\mapsto R_\lambda(x)$ is monotone increasing, so $\rho(\lambda)$ is also monotone increasing in $\lambda$. It is well-known that $\rho(\lambda)\in \Q/\Z$ if and only if $R_\lambda$ has periodic points.   See \cite{Herman1979,Nitecki1971}.

\begin{proof} The existence and uniqueness of $\frak{S}_\lambda$ are proved in Lemmas~\ref{lem:rationalSturm} and~\ref{lem:irrationalSturm} below. By~\cite{Veerman1989}, $\dim (\Gamma)=0$, so it suffices to show that $\overline{\Gamma}\setminus \Gamma$ is countable. Indeed, if $\lambda\in \overline{\Gamma}\setminus \Gamma$, then $\frak{S}_\lambda$ is a periodic measure and the support is not contained in the interior of $C_\lambda$, hence either $\lambda\mod 1$, or $\lambda+1/q \mod 1$ is periodic under $T$. Thus $\overline{\Gamma}\setminus \Gamma$ is countable. Consequently, $\dim (\overline{\Gamma})=0$.
\end{proof}
In the following two lemmas, we treat separately the cases of rational and irrational rotation numbers.

\begin{lemma}\label{lem:rationalSturm}
Suppose that the rotation number $\rho(R_\lambda)$ is rational. Then $R_\lambda$ has a unique invariant probability measure $\frak{S}_\lambda$ supported in $C_\lambda$, and the support of this measure is a periodic orbit of $T$.
%More precisely, exactly one of the following holds:
%\begin{itemize}
%\item[{\rm (1)}] $R_\gamma$ has a unique periodic orbit, this orbit is contained in $C_\gamma$ and intersect $\partial C_\gamma$;
%\item[{\rm (2)}] $R_\gamma$ has exactly two  periodic orbits, one of which is completely contained in $\text{\rm int}(C_\gamma)$ and the other %intersects $\T\setminus C_\gamma$.
%\end{itemize}%contains a unique periodic orbit of $R_\gamma$ and $\frak{S}_\gamma$ is supported on this orbit.
\end{lemma}
\begin{proof} Since $\rho(R_\lambda)$ is rational, all invariant probability measures of $R_\lambda$ are supported on periodic points. So it suffices to show that $R_\gamma$ has a unique periodic orbit contained in $C_\lambda$. Let $p$ be the minimal positive integer such that $p\cdot \rho(R_\gamma)=0\mod 1$. Then each periodic point of $R_\gamma$ has period $p$. %by Proposition \ref{prop:circlemap}.
Let us say that a periodic orbit of $R_\gamma$ is of
\begin{itemize}
\item type I, if the orbit is contained in the interior of $C_\gamma$;
\item type II, if the orbit intersects $\T\setminus C_\gamma$;
\item type III, if the orbit is contained in $C_\gamma$ but intersects $\partial C_\gamma$.
\end{itemize}
A periodic point is said of type I (resp. II, III) if its orbit is of that type.
Let us make the following remarks. If $y$ is a type I periodic point, then $(R_\gamma^p)'=q^p$ in a neighborhood of $y$, so $y$ is two-sided repelling. If $y$ is a type II periodic point, then $(R_\gamma^p)'=0$ in a neighborhood of $y$, so $y$ is two-sided attracting. Since both $\pi(\gamma)$ and $\pi(\gamma+q^{-1})$ are mapped by $R_\gamma$ to the same point $\pi(q\gamma)$, only one of them can be periodic. So there can be at most one periodic orbit of type III, which contains either $\pi(\gamma)$ or $\pi(\gamma+q^{-1})$, and each point in this orbit is attracting from one-side and repelling from the other side.

First assume that there exists a periodic orbit $\mathcal{O}$ of type III.  Then we show that $\mathcal{O}$ is the only periodic orbit of $R_\gamma$.  Without loss of generality, assume that the orbit contains $\pi(\gamma)$. Since $R_\gamma(x)=R_\gamma(\pi(\gamma))$ for all $x\in \T\setminus C_\gamma$, there exists no type II periodic point. There cannot be periodic points of type I either, otherwise, there would exist an arc $J=[a,b]$ with $a\in \mathcal{O}$ and $b$ a periodic point of type I and with no periodic point in the interior of $J$.  This is impossible because $a$ is repelling from the right hand side and $b$ is repelling from the left hand side (in fact from  both sides).

Next assume that there is no periodic orbit of type III, that is to say, all periodic points are of type I or II. Then, by the above remarks, each periodic point is either attracting (from both sides) or repelling from both sides. In particular, there are only finitely many periodic points.%~\marginpar{added} 
Note that if $a$ and $ b$ are two adjacent periodic points, then one of them must be attracting and the other repelling. Thus the number of periodic points of type I is the same as that of type II.
Since $R_\gamma$ is constant on $\T\setminus C_\gamma$, there is only one periodic orbit of type II. It follows that there exists exactly one periodic orbit of type I and exactly one of type II.

We have thus proved that $R_\gamma$ has exactly one periodic orbit contained in $C_\gamma$. 
%By Proposition \ref{prop:circlemap},  all of its invariant probability measures are supported on periodic orbits, it follows that $R_\gamma$ has a %unique invariant measure supported in $C_\gamma$ and it is actually supported on the unique periodic orbit in $C_\gamma$.
\end{proof}

\begin{lemma} \label{lem:irrationalSturm}
Suppose that $\rho(R_\gamma)$ is irrational. Then there is a unique $T$-invariant Borel probability measure
%, denoted by $\frak{S}_\gamma,$
supported in $C_\gamma\subset \T$. 
%Moreover, $T^n(\pi(\gamma))\in C_\gamma$ for all $n\ge 0$.
%$$\supp(\frak{S}_\gamma)=\{x\in C_\gamma: T^n(x)\in C_\gamma \text{ for all } n\ge 0\}\ni \pi(\gamma).$$
%$T^n(\gamma\mod 1)\in C_\gamma$ for all $n\ge 1$.
%\in \supp (\frak{S}_\gamma)$.
\end{lemma}
\begin{proof} By a classical theorem of Poincar\'e, there exists a monotone continuous circle map of degree one such that $h\circ R_\gamma(x)= h(x)+\rho(R_\gamma)\mod 1$.
%The existence of such a $h$ is ensured by Proposition \ref{prop:circlemap}.
	Let $$
	E=\{x\in \T: \# h^{-1}(x)>1\}, \quad E'=\bigcup_{x\in E} h^{-1}(x).$$
	 Since $h$ is monotone,   $\{h^{-1}(x):x\in E\}$ is a disjoint family of non-degenerate (closed) arcs in $\T$. So $E$ is countable.
Note that $E'\supset \T\setminus C_\gamma$, since $R_\gamma$ is constant in $\T\setminus C_\gamma$.

Let $\mu$ be a $T$-invariant measure supported by $C_\gamma$.
Then $\mu$ is a $R_\gamma$-invariant probability measure. Let
$\mu$ be an arbitrary $R_\gamma$-invariant probability measure. Let us prove that $\mu(I)=|h(I)|$ for any arc $I\subset \T$. This will imply that $R_\gamma$ is uniquely ergodic and $\text{supp} (\mu)=\T\setminus \text{int}(E')\subset C_\gamma$.
%Let us prove that $R_\gamma$ is uniquely ergodic and the unique invariant probability measure is supported in $C_\gamma$. To this end, let $\mu$ be %an arbitrary invariant probability measure of $R_\gamma$. It suffices to show that $\mu(I)=|h(I)|$ for each arc in $\T$ and $\mu$ is supported in %$C_\gamma$.
Indeed, the image measure $h_*(\mu)$ is an invariant probability measure of the rigid rotation $x\mapsto x+\rho(R_\gamma)\mod 1$, which is necessarily the Lebesgue measure, for $\rho(R_\gamma)$ is irrational. Observe that for each arc $I\subset \T$,
$$ I\subset h^{-1}(h(I)), \qquad I\supset h^{-1}(h(I)\setminus E).$$
 Thus we get $\mu(I)=|h(I)|$ because
 \[
 \mu(I)\le h_*\mu(h(I))=|h(I)|, \qquad \mu(I)\ge h_*\mu(h(I)\setminus E)=|h(I)|.
 \]
%Since $R_\gamma$ is constant on $\T\setminus C_\gamma$, so is $h$. Thus $\mu(\T\setminus C_\gamma)\le h_*\mu (h(\T\setminus C_\gamma))=0.$
%
%
%Finally, let us prove $x_n:=T^n(\pi(\gamma))\in C_\gamma$ for all $n\ge 0$. Otherwise, there exists a minimal $n_0\ge 1$ such that $x_{n_0}\in %\T\setminus C_\gamma$. But then $$
%   R_\gamma(x_{n_0})=R_\gamma(\pi(\gamma))=T(\pi(\gamma))=x_1.
%$$
%%=R_\gamma (\pi(\gamma))$.
%This implies that $x_{n_0}$ is  a periodic point
%of $R_\gamma$ of period $n_0$, contradicting with the assumption $\rho(\gamma)\not\in \Q$.
%So $\mu$ is supported in $\T\setminus E'\subset C_\gamma$. Therefore $\mu$ is uniqueSince $h$\T\setminus C_\gamma$, it follows that $\mu(\T\setminus %C_\gamma)=0$, i.e. $\mu$ is supported in $C_\gamma$.
%
%It is well-known (Poincar\'e's theorem) that $R_\rho$ is semi-conjugate to the rigid rotation $x\mapsto x+\rho\mod 1$  via a monotone circle map of %degree one. As the rigid rotation has the Lebesgue measure as the unique invariant Borel probability measure, the statement follows.
\end{proof}
%\subsection{Proof of Proposition~\ref{prop:lbarb}}
\section {Appendix B: Computation of $\beta(c)$ and $\gamma(c)$}\label{sec:compute}

The theory developed in Section \ref{sec:max}
and Section \ref{sect:PStoS} allows us to compute
$\beta(c)$ and then $\gamma(c)$ for a very large set of $c$'s.
The computation is based on Proposition \ref{prop:bouschsturm}, Proposition \ref{lem:Jenkinsonlemma} and Theorem \ref{Pre-S}. The method is computer-aided,
but the results are exact because the computer is only used to test
the signs of two quantities,  which don't need to be  exactly computed. We just consider the case $q=2$.

\subsection{Algorithm}\, \ \

Let $\vartheta:=\{s_1, \cdots, s_m\}\subset [0,1)$ be an $m$-periodic
cycle which is contained in some closed semi-circle.  Then the measure $\mu_{\vartheta}: = \frac{1}{m}\sum_{j=1}^m \delta_{s_j}$
is a Sturmian measure. Let
$$
\Lambda_{\vartheta} := [s_{\max} - 1/2, s_{\min}] \quad \mod 1
$$
where $s_{\min} := \min s_j$ and $s_{\max} :=\max s_j$.
Then for any $\lambda \in \Lambda_\vartheta$, the semi-circle $C_\lambda$
contains the support $\vartheta$ of the Sturmian measure $\mu_{\vartheta}$. We emphasize that each $C_\lambda$ contains a unique Sturmian measure, the same measure $\mu_{\vartheta}$ for all $\lambda \in \Lambda_{\vartheta}$.

Given a parameter $c \in [0,1)$, put $b=1/2-c$. Suppose
\begin{equation}\label{pm0}
\lambda_1, \lambda_2 \in (b, b+1/2) \cap \Lambda_{\vartheta}
\end{equation}
 such that
\begin{equation}\label{plus-minus}
v_{c}(\lambda_1) <0 < v_{c}(\lambda_2).
\end{equation}
Then there exists a unique number $\lambda^*$ between $\lambda_1$ and $\lambda_2$ %, and contained in $(1/2-c, 1-c)$
such that $\nu_{\lambda^*}(c) =0$. Therefore the Sturmian measure with support in $C_{\lambda^*}$,
which is $\mu_{\vartheta}$, is the maximizing measure for $f_c$. Thus
\begin{equation}\label{bc}
\beta(c) = \frac{1}{m} \sum_{j=1}^m f_c(s_j)= \frac{1}{m} \log \prod_{j=1}^m |\cos \pi (s_j + c)|+\log 2.
\end{equation}

In practice, we can take as $\lambda_1, \lambda_2$ the end points
of the interval $(b, b+1/2) \cap \Lambda_{\vartheta}$.
We are happy that we don't need to know what $\lambda^*$ is exactly.
See
Table~\ref{tab-2} for the values of $\beta(c)$ for specific $c$'s.

Since $c \mapsto v_c(\lambda)$ is continuous, for given $\lambda_1,\lambda_2$, (\ref{plus-minus}) define an open set of $c$. Thus, if (\ref{plus-minus}) holds, then
%is verified it is still verified under a small perturbation of %$c$ and
%maintain the condition that $\lambda_1, \lambda_2 \in (b(c), %b(c)+1/2) \cap \Lambda_{\vartheta}$ for these perturbed $c$'s. %It follows that
the formula (\ref{bc}) holds not only for $c$ but also on a neighbourhood of $c$. In particular, $\beta(\cdot)$ is analytic at $c$. For a given cycle, there is an interval $[c_*, c^*]$
on which (\ref{bc}) holds.
These intervals are shown in
Table~\ref{tab-1}.

The graph of $\beta(\cdot)$ is shown in Figure \ref{fig:beta}.

\begin{figure}[htb]
	\centering
	\includegraphics[width=0.85\linewidth]{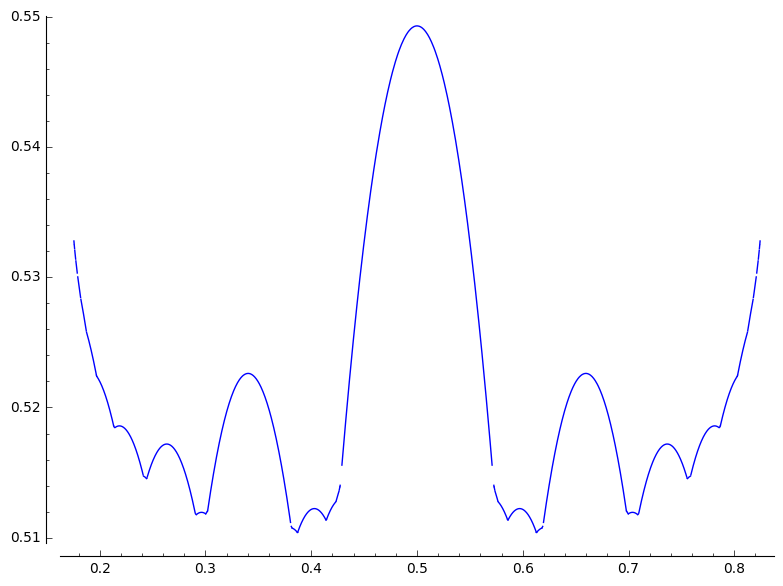}
	\caption{The graphs of $\beta(c)$.}
	\label{fig:beta}
\end{figure}

\subsection{First time leaving $C'_\lambda$} \, \ \

Let $0\le \lambda \le \frac{1}{2}$. Then the map $\tau_\lambda: [0,1)
	\to C'_\lambda:= [\lambda, \lambda +\frac{1}{2})$ is defined by
	$$
	     \tau_\lambda(x) = \frac{x+1}{2}\chi_{[0, 2 \lambda)}(x)
	     + \frac{x}{2}\chi_{(2 \lambda, 1)}(x).
	$$
See Figure \ref{fig:T} for the branch $T|_{C_\lambda}$.
See Figure \ref{fig:01} for the graphs of $e_0$ and $e_{1/4}$.
%If $\frac{1}{2}< \lambda <1$, the map $\tau_\lambda: [0,1]
%\to C_\lambda:= [0, \lambda -\frac{1}{2}) \cup [\lambda, 1]$ is %defined by
%$$
%\tau_\lambda(x) = \frac{x}{2}\chi_{[0, 2\lambda-1)}(x)
%+ \frac{x+1}{2}\chi_{[2\lambda-1, 1)}(x).
%$$
\begin{figure}[htb]
	\centering
	\includegraphics[width=0.45\linewidth]{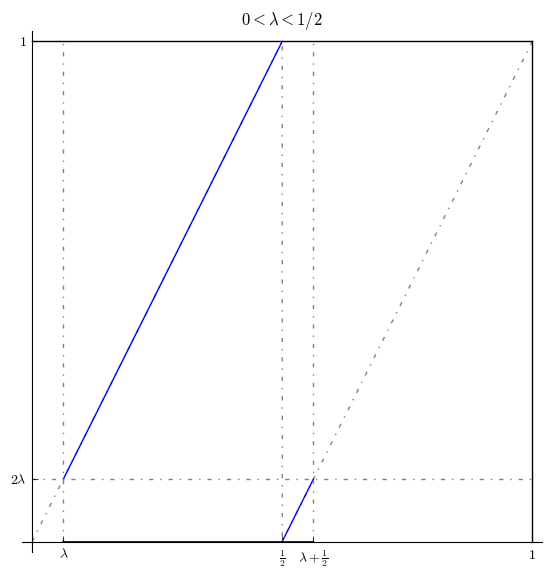}
	\caption{The branch  $T|_{C_\lambda}$
		% for $0<\lambda<\frac{1}{2}$ (left) and  $\frac{1}{2}<\lambda<1$ %(right)
	}
	\label{fig:T}
\end{figure}
\iffalse
It is easy to find the explicit expressions of $e_0$:% and $e_{1/2}$:	
$$
e_0(x) = \sum_{n=1}^\infty n \chi_{[2^{-(n+1)}, 2^{-n})}(x).
$$
\fi
%$\lambda =\frac{1}{2}$, $\tau_{\frac{1}{2}}: [0,1] \to [\frac{1}{2}, 1]$ %defined by $\tau_{\frac{1}{2}}(x) = \frac{x}{2}+\frac{1}{2}$.
%$$
%    e_{\frac{1}{2}}(x) = \sum_{n=1}^\infty n \chi_{[2^{-1}+ \cdots %+2^{-n}, 2^{-1}+ \cdots +2^{-n} +2^{-(n+1)})} (x).
%$$

 \begin{figure}[htb]
	\centering
	\includegraphics[width=0.45\linewidth]{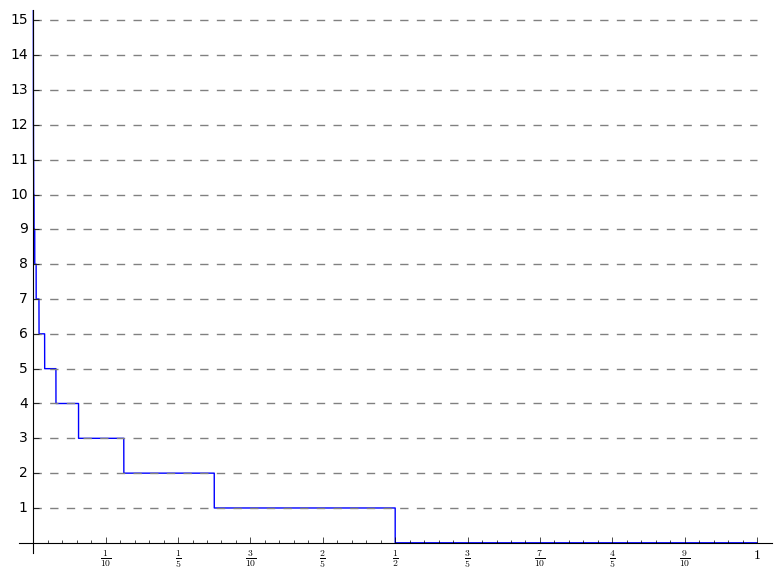}
	\includegraphics[width=0.45\linewidth]{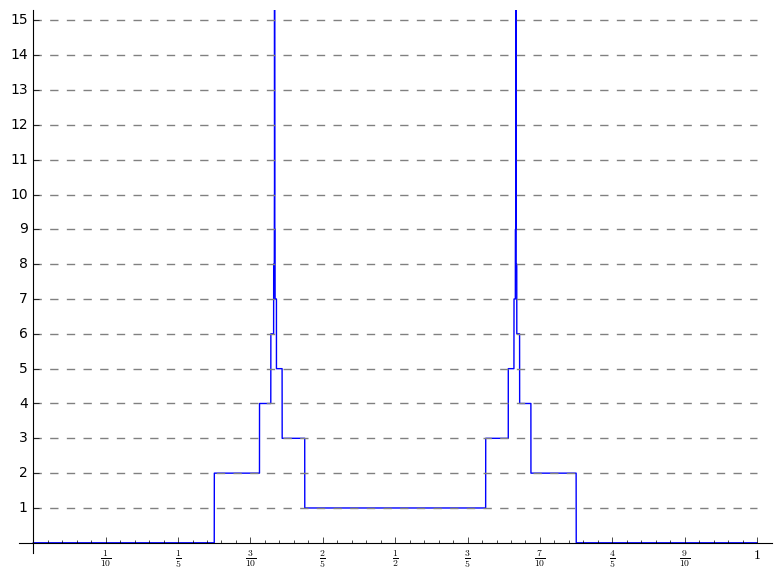}
	\caption{The graphs of $e_0$ and $e_{1/4}$.}
	\label{fig:01}
\end{figure}

\iffalse
It is also easy to find the expressions of $e_{3/4}$:
$$
e_{\frac{3}{4}}(x) = \sum_{n=1}^\infty
\left(\chi_{[0,  2^{-(n+1)})} (x) + \chi_{[1-2^{-(n+1)}, 1)}(x)
\right).
$$

\begin{figure}[htb]
	\centering
	\includegraphics[width=0.45\linewidth]{8-32}
	\includegraphics[width=0.45\linewidth]{24-32}
	\caption{The graphs of $e_{1/4}$ and $e_{3/4}$.}
	\label{fig:03}
\end{figure}
\fi

%$\lambda =\frac{1}{2}$, $\tau_{\frac{1}{2}}: [0,1] \to [\frac{1}{2}, 1]$ %defined by $\tau_{\frac{1}{2}}(x) = \frac{x}{2}+\frac{1}{2}$.
Let us look at $e_{1/4}$.  Observe  that $e_{1/4}$
is symmetric with respect to $1/2$, i.e. $e_{1/4}(x) = e_{1/4}(1-x)$
for a.e. $x \in [0,1]$. Indeed,
 $C_{1/4}$ is a union of two
intervals of length $1/4$ which are symmetric with respect to $1/2$, %i.e.
%$$
%       C_{1/4} = [0, 2^{-2}] +\{2^{-1}, 2^{-2}\}
%$$
and $\tau_{1/4} x = x/2+1/2$ if $x \in [0, 1/2)$
and $\tau_{1/4} x = x/2$ if $x \in [1/2, 1)$ so that
it maps two symmetric intervals to two symmetric intervals.
%It is easy to see
%that
%\begin{eqnarray*}
%   \tau_{1/4}(C_{1/4})
%   &= & [0, 2^{-3}] +\{ 2^{-3}+2^{-1}, 2^{-2}\},
%   \\
%    \tau^2_{1/4}(C_{1/4}) &=& [0, 2^{-4}] +\{ 2^{-4}+2^{-2}, %2^{-3}+2^{-1}\}.
%\end{eqnarray*}
%By induction, we can prove that for $n\ge 0$ we have
%\begin{eqnarray*}
%    \tau_{1/4}^{2n}(C_{1/4}) &=& [0, 2^{-(2n+2)}]+ \{2^{-1} + 2^{-3} + %\cdots + 2^{-(2n+1)}, 2^{-2} + 2^{-4} + \cdots + 2^{-(2n+2)}\}
%\\
%\tau_{1/4}^{2n+1}(C_{1/4}) &=& [0, 2^{-(2n+3)}]+ \{2^{-1} + 2^{-3} + %\cdots + 2^{-(2n+3)}, 2^{-2} + 2^{-4} + \cdots + 2^{-(2n+2)}\}.
%\end{eqnarray*}
%$$
%e_{\frac{1}{4}}(x) = \sum_{n=1}^\infty .
%$$

%$\lambda =\frac{3}{4}$, $\tau_{\frac{3}{4}}:
%[0,1] \to [0, \frac{1}{4}]\cup [\frac{3}{4}, 1]$ defined by $$\tau_{\frac{3}{4}}(x) = \frac{x}{2}1_{[0, \frac{1}{2}]}(x) +
%\left(\frac{x}{2}+\frac{1}{2}\right)1_{[\frac{1}{2}, 1]}(x).
%$$
\medskip
\iffalse
{\bf Example 0.} Notice that $\Lambda_{\{0\}} =[1/2, 1] \ni \frac{3}{4}$
and $C_{3/4} = [-1/4, 1/4] \mod 1$. Also notice that $e_{3/4}$
is an even function on $[-1/4, 1/4]$ and $f_0'$ is an odd function. Then we have $\nu_0(3/4)=0$.   Thus  $$
\beta(0) = \log (\cos \pi \cdot 0) =0.
$$
This confirms the trivial result $\beta(0) =0$, but indicates that $3/4$
is the corresponding $\lambda$ such that $\nu_0(\lambda)=0$.
\fi

\subsection{Some examples}\, \ \

{\bf Example 1.  $\mu_{\{1/3,2/3\}}$ is maximizing  for $f_{1/2}$} and
	$$
	\beta(1/2) = \frac{1}{2}\log |\sin(\pi/3) \sin (2\pi/3)|+\log 2
	=\log\sqrt{3}=0.549306144334055.
	$$
This is known to Gelfond \cite{Gelfond1968}. The following is another proof. Recall that in this case $$
(b, b+1/2) \cap \Lambda_{\vartheta} =(0, 1/2)
\cap [1/6,1/3] = [1/6,1/3],$$
 which contains $1/4$. As have noticed above, the function
$e_{1/4}$ is symmetric. On the other hand,  $f_{1/2}$  is anti-symmetric about $1/2$. In other words, we have
   \begin{eqnarray*}
   	e_{1/4}(1/2-x) &= & \ \  e_{1/4}(1/2+x),\\
f'_{1/2}(1/2-x) & = & - f'_{1/2}(1/2 +x).
\end{eqnarray*}
It follows that $\nu_{1/2}(1/4) =0$.  Thus
 $\mu_{\{1/3,2/3\}}$ is the maximizing measure for $f_{1/2}$. It
 follows that $$
 \gamma(1/2) = \frac{\beta(1/2)}{\log 2} = \frac{\log 3}{\log 4}
 =0.79248125036058.
 $$
 \medskip

  {\bf Example 2.  $\mu_{\{7/15,14/15, 13/15, 11/15\}}$ is maximizing  for $f_{1/4}$} and
 \begin{eqnarray*}
\beta(1/4) &=& \frac{1}{4}\log \prod_{j=0}^3|\cos\pi (2^j\cdot 7/15 +1/4)|+\log 2\\
&=&  \frac{1}{4}\log \left|\cos\frac{43\pi}{60} \cos\frac{11\pi}{60}
\cos\frac{7\pi}{60}  \cos\frac{59\pi}{60}\right|+\log2  \\
&=&0.51585926722389.
\end{eqnarray*}
$$
     \gamma(1/4) =\frac{\beta(1/4)}{\log 2}=0.74422760662052.
$$
In fact, in this case $$
 (b, b+1/2) \cap \Lambda_{\vartheta} =(1/4, 3/4)
 \cap [13/30,14/30] =  [13/30,14/30].$$
 Numerical computation shows that $v_{1/4}(13/30) v_{1/4}(14/30) <0$.
 Thus
 $\mu_{\{7/15,14/15, 13/15, 11/15\}}$ is the maximizing measure for $f_{1/4}$.
 \medskip

  {\bf Example 3.  $\mu_{\{1/15,2/15, 4/15, 8/15\}}$ is maximizing  for $f_{3/4}$} and
  %~\marginpar{I suggest that we do not include this example, as it follows %from Ex 2 by symmetry}
  \begin{eqnarray*}
 \beta(3/4) = % & = &
 \frac{1}{4}\log \prod_{j=0}^3|\cos\pi (2^j\cdot 1/15 +3/4)| +\log 2= %\\
 %&=& \frac{1}{4}\log \left|\cos \frac{49\pi}{60}
% \cos \frac{53\pi}{60} \cos \frac{\pi}{60} \cos \frac{17\pi}{60}
 %\right| \\
% &=&
 0.515859267223890.
 \end{eqnarray*}
 In this case $$
 (b, b+1/2) \cap \Lambda_{\vartheta} =(-1/4, 1/4)
 \cap [1/30,2/30] =  [1/30,2/30].$$
 Numerical computation shows that $\nu_{3/4}(1/30) \nu_{3/4}(2/30) <0$.
 Thus
 $\mu_{\{1/15,2/15, 4/15, 8/15\}}$ is the maximizing measure for $f_{3/4}$.

 We can get immediately the value of $\beta (3/4)$ from {\bf Example 2}, by symmetry (Proposition \ref{symmetry}). But we would like to remark the maximizing measures for
 $f_{1/4}$ and $f_{3/4}$ are different.

 {\bf Example 4.  $\mu_{\{3/7,6/7, 5/7\}}$ is maximizing  for $f_{1/3}$} and
\begin{eqnarray*}
 \beta(1/3) &=& \frac{1}{3}\log \prod_{j=0}^2|\cos\pi
  (2^j\cdot 3/7 +1/3)|+\log 2\\
 &=& \frac{1}{3}\log \left(\cos \frac{16 \pi}{21} \cos \frac{4\pi}{21} \cos \frac{\pi}{21}\right)+\log 2\\
 &=& 0.522266412324137
 \end{eqnarray*}
and
\[
   \gamma(1/3) =\frac{\beta(1/3)}{\log 2} =0.81510337231218.
\]
We have only to check $\nu_{1/3}(5/14) \nu_{1/3}(6/14) <0$.

%\iffalse
\subsection{Numerical results}

See
Table~\ref{tab-2} for the values of $\beta(c)$ for specific $c$'s.
The graph of $\beta(\cdot)$ is already shown in Figure \ref{fig:beta}.

We obtain these numerical and graphic results only
using  periodic Sturmian measures of period $ \le 13$.
There are totally $57$ Sturmian cycles of period
$m=2,3,\dots,13$. Thus we find $57$ $\lambda$-intervals $[s_{\max}, s_{\min} -1/2]$ and $57$ $c$-intervals of parameter $[c_*, c^*]$.
These intervals are shown in
Table~\ref{tab-1}.
Notice that both   $\beta(c)$ and $\gamma(c)$ are computed only for $c$ or $ 1-c <0.175633988226123$. More results can be obtained if
we consider periodic Sturmian measures of period $ \ge 14 $.

 For any Sturmian cycle $\vartheta=\{s_1, \cdots, s_m\}$, there is an interval
 $\Lambda_{\vartheta} = [s_{\max} -1/2, s_{\min}]$ of $\lambda$
 and an interval
 $[c_*,c^*]$ of $c$. The value of $\beta(c)$ for $c\in [c_*, c^*]$
 is expressed by the formula (\ref{bc}).
%\fi

 \begin{center}
 	\begin{threeparttable}
 		\small
 		\caption{Values of $ \beta(c) $ and $ \gamma(c) $ for specific $ c $'s}
 		\label{tab-2}
 		\centering
 		\begin{tabular}{@{}llllll@{}}
 			\toprule
 			$c$ & $\beta(c)$ & $ \gamma(c)$ & $ c $ & $ \beta(c) $ & $\gamma(c)$ \\ \midrule		
 			$1/2$ &  $\log(\sqrt{3})$ & $ \log 3/\log 4$  &   $7/18$ &  $0.51079$ &  $ 0.73691 $             \\
 			$1/3$ &  $0.52227$ &	$0.75347$&	    $4/19$ &  $0.51949$ &  $0.74947$						    \\
 			$1/4$ &  $0.51586$ &	$0.74423$&		$5/19$ &  $0.51719$	&  $0.74615$					    \\
 			$1/5$ &  $0.52201$ &   $0.75310$&      $6/19$ &  $0.51830$	&  $0.74775$						\\
 			$2/5$ &  $0.51217$ &   $0.73890$&      $7/19$ &  $0.51701$	&  $0.74589$						\\
 			$2/7$ &  $0.51354$ &	$0.74088$&	    $8/19$ &  $0.51252$	&  $0.73941$					    \\
 			$3/7$ &  $0.51515$ &	$0.74321$&		$9/19$ &  $0.54474$	&  $0.78589$					\\		
 			$3/8$ &  $0.51406$ &	$0.74163$&		$7/20$ &  $0.52195$	&  $0.75302$						\\
 			$2/9$ &  $0.51848$ &	$0.74802$&		$9/20$ &  $0.53272$	&  $0.76855$						\\
 			$4/9$ &  $0.52879$ &	$0.76288$&		$4/21$ &  $0.52489$	&  $0.75725$						\\
 			$3/10$&  $0.51184$ &	$0.73843$&		$5/21$ &  $0.51576$	&  $0.74408$						\\
 			$2/11$&  $0.52852$ &	$0.76250$&		$8/21$ &  \multicolumn{2}{c}{**}			\\
 			$3/11$&  $0.51655$ &	$0.74523$&		$5/22$ &  $0.51802$	&  $0.74735$				\\
 			$4/11$&  $0.51875$ &	$0.74840$&		$7/22$ &  $0.51910$	&  $0.74891$				\\
 			$5/11$&  $0.53562$ &	$0.77273$&		$9/22$ &  $0.51196$	&  $0.73860$				\\
 			$5/12$&  $0.51185$ &	$0.73844$&		$5/23$ &  $0.51857$	&  $0.74814$				\\		
 			$3/13$&  $0.51748$ &	$0.74657$&		$6/23$ &  $0.51714$	&  $0.74608$				\\
 			$4/13$&  $0.51496$ &	$0.74293$&		$7/23$ &  $0.51329$	&  $0.74052$				\\
 			$5/13$&  $0.49827$ &	$0.71885$&		$8/23$ &  $0.52222$	&  $0.75340$				\\
 			$6/13$&  $0.53952$ &	$0.77837$&		$9/23$ &  $0.51124$	&  $0.73756$				\\
 			$3/14$&  $0.51844$ &	$0.74795$&		$10/23$ &  $0.52092$	&  $0.75153$				\\
 			$5/14$&  $0.52061$ &	$0.75108$&		$11/23$ &  $0.54619$	&  $0.78799$				\\
 			$7/15$&  $0.54197$ &	$0.78190$&		$5/24$ &  $0.52015$	&  $0.75042$				\\
 			$7/16$&  $0.52326$ &	$0.75491$&		$7/24$ &  $0.51179$	&  $0.73836$				\\
 			$3/17$&  $0.53203$ &	$0.76756$&		$11/24$ &  $0.53782$	&  $0.77591$				\\		
 			$4/17$&  $0.51651$ &	$0.74516$&		$6/25$ &  $0.51517$	&  $0.74324$				\\
 			$5/17$&  $0.51191$ &	$0.73853$&		$7/25$ &  $0.515168$	&  $0.74323$				\\
 			$6/17$&  $0.52148$ &	$0.75234$&		$8/25$ &  $0.51966$	&  $0.74971$				\\
 			$7/17$&  $0.51167$ &	$0.73818$&		$9/25$ &  $0.51987$	&  $0.75001$				\\
 			$8/17$&  $0.54360$ &	$0.78425$&		$11/25$ &  $0.52534$	&  $0.75789$				\\
 			$5/18$&  $0.51567$ &	$0.74396$&		$12/25$ &  $0.54667$	&  $0.78868$				\\								
 			\bottomrule
 		\end{tabular}
 		\begin{tablenotes}
 			\tiny
 			\item \textbf{**} We don't compute $\beta(c)$ and $\gamma(c)$
 			if the parameter $c$ doesn't belong  to any of the intervals in Table~\ref{tab-1}.
 		\end{tablenotes}
 	\end{threeparttable}
 \end{center}

\newpage

\begin{center}
\small
\begin{longtable}{@{}clll@{}}
	 	\caption{Valid intervals $[c_*, c^*]$}\label{tab-1}\\
		
		\toprule \multicolumn{1}{c}{\textbf{Period}} & \multicolumn{1}{c}{\textbf{$s_{\max}-\frac{1}{2}$}} & \multicolumn{1}{c}{\textbf{$s_{\min}$}} &
		\multicolumn{1}{c}{\textbf{$ [c_*,c^*] $}}\\ \midrule
		\endfirsthead

		\multicolumn{4}{c}%
		{{\scshape \tablename\ \thetable{}} -- Continued from  previous page} \\
    
		\toprule \multicolumn{1}{c}{\textbf{Period}} & \multicolumn{1}{c}{\textbf{$s_{\max}-\frac{1}{2}$}} & \multicolumn{1}{c}{\textbf{$s_{\min}$}} &
  		\multicolumn{1}{c}{\textbf{$ [c_*,c^*] $}}\\ \midrule
		\endhead
		
		\hline \multicolumn{4}{r}{{(Continued on next page)}} \\
		\endfoot
    
		\bottomrule
		\endlastfoot
		
		$2$ & $1/6$ & $1/3$ &  $[0.428133329021334,0.571866670978666]$ \\
		$ 3 $ & $1/14 $& $ 1/7 $ & $ [0.619203577131485,0.697872156658965] $ \\
		$ 3 $ & $5/14 $& $ 3/7 $ & $ [0.302127843341035,0.380796422868515] $ \\
		$ 4 $ & $1/30 $& $ 1/15 $ & $ [0.709633870795466,0.755421357085333] $ \\
		$ 4 $ & $13/30 $& $ 7/15 $ & $ [0.244578642914667,0.290366129204534] $ \\
		$ 5 $ & $1/62 $& $ 1/31 $ & $ [0.758710839860046,0.785842721390351] $ \\
		$ 5 $ & $29/62$& $15/31 $ & $ [0.214157278609649,0.241289160139954] $ \\
		$ 5 $ & $9/62$& $ 5/31$ & $ [0.586141644350735,0.612800854796395] $ \\
		$ 5 $ & $21/62$& $11/31$ & $[0.387199145203605,0.413858355649265] $ \\
		$ 6 $ & $1/126$& $1/63$ & $[0.786809543609523,0.802555581755556] $ \\
		$ 6 $ & $61/126$& $31/63$ & $[0.197444418244444,0.213190456390477] $ \\
		$ 7 $ & $1/254$& $1/127$ & $[0.803225220690394,0.812352783425512] $ \\
		$ 7 $ & $125/254$&$63/127$ & $[0.187647216574488,0.196774779309606] $ \\
		$ 7 $ & $17/254$& $9/127$ & $[0.699811031164904,0.708527570112261] $ \\
		$ 7 $ & $109/254$& $55/127$ & $[0.291472429887739,0.300188968835096] $ \\
		$ 7 $ & $41/254$& $21/127$ & $[0.576825192903727,0.585555905085145] $ \\
		$ 7 $ & $85/254$& $43/127$ & $[0.414444094914855,0.423174807096273] $ \\
		$ 8 $ & $1/510$& $1/255$ & $[0.812634013261438,0.817780420556863] $ \\
		$ 8 $ & $253/510$& $127/255$ & $[0.182219579443137,0.187365986738562] $ \\
		$ 8 $ & $73/510$& $37/255$ & $[0.613186931037909,0.617835298917647] $ \\
		$ 8 $ & $181/510$& $91/255$ & $[0.382164701082353,0.386813068962091] $ \\
		$ 9 $ & $1/1022$& $1/511$ & $[0.818062650175864,0.820724099383431] $ \\
		$ 9 $ & $509/1022$&$255/511$ & $[0.179275900616569,0.181937349824136] $ \\
		$ 9 $ & $33/1022$& $17/511$ & $[0.755812148539074,0.758473597746640] $ \\
		$ 9 $ & $477/1022$& $239/511$ & $[0.241526402253360,0.244187851460926] $ \\
		$ 9 $ & $169/1022$& $85/511$ & $[0.576825192903727,0.585555905085145] $ \\
		$ 9 $ & $341/1022$& $171/511$ & $[0.423502938487411,0.426164387694977] $ \\
		$10 $ & $1/2046$& $1/1023$ & $[0.821196509738417,0.822528540248941] $ \\
		$10 $ & $1021/2046$ & $511/1023$ & $[0.177471459751059,0.178803490261583] $ \\
		$10 $ & $145/2046$ & $73/1023$ & $[0.698241698854594,0.699698607225480] $ \\
		$10 $ & $877/2046$& $439/1023$ & $[0.300301392774520,0.301758301145406] $ \\
		$11 $ & $1/4094$& $1/2047$ & $[0.822722890076930,0.823555816343776] $ \\
		$11 $ & $2045/4094$ & $1023/2047$ & $[0.176444183656224,0.177277109923070] $ \\
		$11 $ & $65/4094$& $33/2047$ & $[0.786058868717432,0.786683563417567] $ \\
		$11 $ & $1981/4094$& $991/2047$ & $[0.213316436582433,0.213941131282568] $ \\
		$11 $ & $273/4094$& $137/2047$ & $[0.708807402099743,0.709432096799878] $ \\
		$11 $ & $1773/4094$& $887/2047$ & $[0.290567903200122,0.291192597900257] $ \\
		$11 $ & $585/4094$& $293/2047$ & $[0.618230337528651,0.619009737929774] $ \\		
		$11 $ & $1461/4094$& $731/2047$ & $[0.380990262070226,0.381769662471349] $ \\
		$11 $ & $681/4094$& $341/2047$ & $[0.572917073341487,0.573541768041622] $ \\
		$11 $ & $1365/4094$& $683/2047$ & $[0.426458231958378,0.427082926658513] $ \\
		$12 $ & $1/8190$& $1/4095$ & $[0.823705054848802,0.824017478548726] $ \\
		$12 $ & $4093/8190$& $2047/4095$ & $[0.175982521451274,0.176294945151198] $ \\
		$12 $ & $1321/8190$& $661/4095$ & $[0.585676663495414,0.585989087195338] $ \\
		$12 $ & $2773/8190$& $1387/4095$ & $[0.414010912804662,0.414323336504586] $ \\
		$13 $ & $1/16382$& $1/8191$ & $[0.824099377662201,0.824366011773877] $ \\
		$13 $ & $8189/16382$& $4095/8191$ & $[0.175633988226123,0.175900622337799] $ \\
		$13 $ & $129/16382$& $65/8191$ & $[0.802834232937408,0.803074203637915] $ \\
		$13 $ & $8061/16382$& $4031/8191$ & $[0.196925796362085,0.197165767062592] $ \\
		$13 $ & $545/16382$& $273/8191$ & $[0.755457525487725,0.755686164238488] $ \\
		$13 $ & $7645/16382$& $3823/8191$ & $[0.244313835761512,0.244542474512275] $ \\
		$13 $ & $1169/16382$& $585/8191$ & $[0.697932644443065,0.698161283193827] $ \\
		$13 $ & $7021/16382$& $3511/8191$ & $[0.301838716806173,0.302067355556935] $ \\
		$13 $ & $2377/16382$& $1189/8191$ & $[0.612842893451498,0.613081331005864] $ \\
		$13 $ & $5813/16382$& $2907/8191$ & $[0.386918668994136,0.387157106548502] $ \\
		$13 $ & $2729/16382$& $1365/8191$ & $[0.572640180643138,0.572864153296945] $ \\
		$13 $ & $5461/16382$& $2731/8191$ & $[0.427135846703055,0.427359819356862] $ \\
\end{longtable}
\end{center}

%  \begin{thebibliography}{99}

\nocite{*}
\bibliographystyle{plain}
\bibliography{SWX1.bib}

\begin{thebibliography}{10}

\bibitem{AHL2017}
Christoph Aistleitner, Roswitha Hofer, and Gerhard Larcher.
\newblock On evil {K}ronecker sequences and lacunary trigonometric products.
\newblock {\em Ann. Inst. Fourier (Grenoble)}, 67(2):637--687, 2017.

\bibitem{ADJR2010}
V.~Anagnostopoulou, K.~Diaz-Ordaz, O.~Jenkinson, and C.~Richard.
\newblock Entrance time functions for flat spot maps.
\newblock {\em Nonlinearity}, 23(6):1477--1494, 2010.

\bibitem{ADJR2012b}
V.~Anagnostopoulou, K.~Diaz-Ordaz, O.~Jenkinson, and C.~Richard.
\newblock The flat spot standard family: variation of the entrance time median.
\newblock {\em Dyn. Syst.}, 27(1):29--43, 2012.

\bibitem{ADJR2012}
V.~Anagnostopoulou, K.~Diaz-Ordaz, O.~Jenkinson, and C.~Richard.
\newblock Sturmian maximizing measures for the piecewise-linear cosine family.
\newblock {\em Bull. Braz. Math. Soc. (N.S.)}, 43(2):285--302, 2012.

\bibitem{Bochi2018}
Jairo Bochi.
\newblock Ergodic opitimization of {B}irkhoff averages and {L}yapunov
  exponents.
\newblock {\em Proc. Int. Cong. Math. 2018 Rio de Janeiro, vol. 3., 1843-1864},
  2018.

\bibitem{Bousch2000}
Thierry Bousch.
\newblock Le poisson n'a pas d'ar\^etes.
\newblock {\em Ann. Inst. H. Poincar\'e Probab. Statist.}, 36(4):489--508,
  2000.

\bibitem{Bousch2001}
Thierry Bousch.
\newblock La condition de {W}alters.
\newblock {\em Ann. Sci. \'Ecole Norm. Sup. (4)}, 34(2):287--311, 2001.

\bibitem{BJ2002}
Thierry Bousch and Oliver Jenkinson.
\newblock Cohomology classes of dynamically non-negative {$C^k$} functions.
\newblock {\em Invent. Math.}, 148(1):207--217, 2002.

\bibitem{Boyd1985}
Colin Boyd.
\newblock On the structure of the family of {C}herry fields on the torus.
\newblock {\em Ergodic Theory Dynam. Systems}, 5(1):27--46, 1985.

\bibitem{BS1994}
Shaun Bullett and Pierrette Sentenac.
\newblock Ordered orbits of the shift, square roots, and the devil's staircase.
\newblock {\em Math. Proc. Cambridge Philos. Soc.}, 115(3):451--481, 1994.

\bibitem{CLT2001}
G.~Contreras, A.~O. Lopes, and Ph. Thieullen.
\newblock Lyapunov minimizing measures for expanding maps of the circle.
\newblock {\em Ergodic Theory Dynam. Systems}, 21(5):1379--1409, 2001.

\bibitem{C2016}
Gonzalo Contreras.
\newblock Ground states are generically a periodic orbit.
\newblock {\em Invent. Math.}, 205(2):383--412, 2016.

\bibitem{CG}
Jean-Pierre Conze and Yves Guivarc'h.
\newblock Croissance des sommes ergodiques et principe variationnel.
\newblock {\em Unpublished preprint}.

\bibitem{DT2005}
C\'ecile Dartyge and G\'erald Tenenbaum.
\newblock Sommes des chiffres de multiples d'entiers.
\newblock {\em Ann. Inst. Fourier (Grenoble)}, 55(7):2423--2474, 2005.

\bibitem{F}
Ai-Hua Fan.
\newblock Weighted {B}irkhoff ergodic theorem with oscillating weights.
\newblock {\em Ergodic Theory and Dynamical Systems}, pages 1--15, 2017.

\bibitem{FK2018}
Aihua Fan and Jakub Konieczny.
\newblock On uniformity of q-multiplicative sequences.
\newblock 2018.
\newblock Preprint. https://arxiv.org/abs/1806.04267v1.

\bibitem{FSW2}
Aihua Fan, J\"{o}rg Schmeling, and Weixiao Shen.
\newblock Multifractal analysis of generalized {T}hue-{M}orse polynomials.

\bibitem{FM1996b}
E.~Fouvry and C.~Mauduit.
\newblock M\'ethodes de crible et fonctions sommes des chiffres.
\newblock {\em Acta Arith.}, 77(4):339--351, 1996.

\bibitem{FM1996a}
E.~Fouvry and C.~Mauduit.
\newblock Sommes des chiffres et nombres presque premiers.
\newblock {\em Math. Ann.}, 305(3):571--599, 1996.

\bibitem{Gelfond1968}
A.~O. Gel'fond.
\newblock Sur les nombres qui ont des propri\'et\'es additives et
  multiplicatives donn\'ees.
\newblock {\em Acta Arith.}, 13:259--265, 1967/1968.

\bibitem{Herman1979}
Michael-Robert Herman.
\newblock Sur la conjugaison diff\'{e}rentiable des diff\'{e}omorphismes du
  cercle \`a des rotations.
\newblock {\em Inst. Hautes \'{E}tudes Sci. Publ. Math.}, (49):5--233, 1979.

\bibitem{JMU2007}
O.~Jenkinson, R.~D. Mauldin, and M.~Urba\'nski.
\newblock Ergodic optimization for noncompact dynamical systems.
\newblock {\em Dyn. Syst.}, 22(3):379--388, 2007.

\bibitem{J2017}
Oliver Jenkinson.
\newblock Ergodic optimization in dynamical systems.
\newblock {\em ArXiv}.

\bibitem{J2006}
Oliver Jenkinson.
\newblock Ergodic optimization.
\newblock {\em Discrete Contin. Dyn. Syst.}, 15(1):197--224, 2006.

\bibitem{Jenkinson2007}
Oliver Jenkinson.
\newblock Optimization and majorization of invariant measures.
\newblock {\em Electron. Res. Announc. Amer. Math. Soc.}, 13:1--12, 2007.

\bibitem{J2008}
Oliver Jenkinson.
\newblock A partial order on {$\times2$}-invariant measures.
\newblock {\em Math. Res. Lett.}, 15(5):893--900, 2008.

\bibitem{Jenkinson2009}
Oliver Jenkinson.
\newblock Balanced words and majorization.
\newblock {\em Discrete Math. Algorithms Appl.}, 1(4):463--483, 2009.

\bibitem{Jenkinson-S_2010}
Oliver Jenkinson and Jacob Steel.
\newblock Majorization of invariant measures for orientation-reversing maps.
\newblock {\em Ergodic Theory Dynam. Systems}, 30(5):1471--1483, 2010.

\bibitem{Konieczny2017}
Jakub Konieczny.
\newblock Gowers norms for the {T}hue-{M}orse and {R}udin-{S}hapiro sequences.
\newblock 2017.
\newblock Preprint. https://arxiv.org/abs/1611.09985.

\bibitem{Mane1985}
Ricardo Ma\~{n}\'{e}.
\newblock Hyperbolicity, sinks and measure in one-dimensional dynamics.
\newblock {\em Comm. Math. Phys.}, 100(4):495--524, 1985.

\bibitem{M1927}
Kurt Mahler.
\newblock The spectrum of an array and its application to the study of the
  translation properties of a simple class of arithmetical functions: Part two
  on the translation properties of a simple class of arithmetical functions.
\newblock {\em Journal of Mathematics and Physics}, 6(1-4):158--163, 1927.

\bibitem{MR2009}
Christian Mauduit and Jo\"el Rivat.
\newblock La somme des chiffres des carr\'es.
\newblock {\em Acta Math.}, 203(1):107--148, 2009.

\bibitem{MR2010}
Christian Mauduit and Jo\"el Rivat.
\newblock Sur un probl\`eme de {G}elfond: la somme des chiffres des nombres
  premiers.
\newblock {\em Ann. of Math. (2)}, 171(3):1591--1646, 2010.

\bibitem{MRS2017}
Christian Mauduit, Jo\"el Rivat, and Andr\'as S\'ark\"ozy.
\newblock On the digits of sumsets.
\newblock {\em Canad. J. Math.}, 69(3):595--612, 2017.

\bibitem{Pl}
V.~A. Pliss.
\newblock On a conjecture of smale.
\newblock {\em Diff. Uravnenija}, 8:268–282, 1972.

\bibitem{Queffelec2018}
Martine Queff\'{e}lec.
\newblock Questions around the {T}hue-{M}orse sequence.
\newblock {\em Unif. Distrib. Theory}, 13(1):1--25, 2018.

\bibitem{Veerman1989}
J.~J.~P. Veerman.
\newblock Irrational rotation numbers.
\newblock {\em Nonlinearity}, 2(3):419--428, 1989.

\end{thebibliography}

%\end{thebibliography}
\end{document}